\documentclass{article}

\usepackage[a4paper,top=3cm,bottom=2cm,left=3cm,right=3cm,marginparwidth=1.75cm]{geometry}

\makeatletter
\newcommand\footnoteref[1]{\protected@xdef\@thefnmark{\ref{#1}}\@footnotemark}
\makeatother
\usepackage{tikz,tkz-euclide}
\usepackage{amsmath,amsfonts,amssymb,amsthm,dsfont,natbib,algorithm,algorithmc}
\usepackage[utf8]{inputenc} 
\usepackage[T1]{fontenc}    
\usepackage{hyperref}       
\usepackage{url}            
\usepackage{booktabs}       
\usepackage{amsfonts}       
\usepackage{nicefrac}       
\usepackage{microtype}      
\usepackage{multicol,caption,soul}
\usepackage{authblk}
\usepackage{enumitem,kantlipsum}
\usepackage{natbib}

\usepackage[normalem]{ulem}
\newtheorem{theorem}{Theorem}
\newtheorem{example}[theorem]{Example}
\newtheorem{definition}[theorem]{Definition}
\newtheorem{lemma}[theorem]{Lemma}
\newtheorem{proposition}[theorem]{Proposition}

\newcommand{\ba}{\[
	\begin{aligned}}
\newcommand{\ea}{
	\end{aligned}
	\]}
\newcommand{\baa}{\begin{equation}
	\begin{aligned}}
\newcommand{\eaa}{
	\end{aligned}
	\end{equation}}

\usepackage{hyperref}

\begin{document}
\tikzstyle{labelnode}=[circle, draw=white, fill=white]

%

  %

\title{Scale Invariant Power Iteration}
\author[1]{Cheolmin Kim}
\author[2]{Youngseok Kim}
\author[1]{Diego Klabjan}
\affil[1]{Department of Industrial Engineering and Management Sciences, Northwestern University}
\affil[2]{Department of Statistics, University of Chicago}
\date{\vspace{-5ex}}

\maketitle

\begin{abstract}
Power iteration has been generalized to solve many interesting problems in machine learning and statistics. Despite its striking success, theoretical understanding of when and how such an algorithm enjoys good convergence property is limited. In this work, we introduce a new class of optimization problems called scale invariant problems and prove that they can be efficiently solved by scale invariant power iteration (SCI-PI) with a generalized convergence guarantee of power iteration. By deriving that a stationary point is an eigenvector of the Hessian evaluated at the point, we show that scale invariant problems indeed resemble the leading eigenvector problem near a local optimum. Also, based on a novel reformulation, we geometrically derive SCI-PI which has a general form of power iteration. The convergence analysis shows that SCI-PI attains local linear convergence with a rate being proportional to the top two eigenvalues of the Hessian at the optimum. Moreover, we discuss some extended settings of scale invariant problems and provide similar convergence results for them. In numerical experiments, we introduce applications to independent component analysis, Gaussian mixtures, and non-negative matrix factorization. Experimental results demonstrate that SCI-PI is competitive to state-of-the-art benchmark algorithms and often yield better solutions.
\end{abstract}

\section{Introduction}

We consider a generalization of power iteration for finding the leading eigenvector of a matrix $A$. Power iteration repeats $x_{k+1} \leftarrow Ax_k / \| Ax_k \|$ until some stopping criterion is satisfied. Since no hyperparameter is required, this update rule is practical yet attains global linear convergence with the rate of $|\lambda_2|/|\lambda_1|$ where $|\lambda_i|$ is the $i^{th}$ largest absolute eigenvalue of $A$. This linear convergence result is analogous to that of gradient descent for convex optimization. Therefore, many variants including coordinate-wise \citep{lei2016coordinate}, momentum \citep{xu2018accelerated}, online \citep{boutsidis2015online, garber2015online}, stochastic \citep{oja1982simplified}, stochastic variance-reduced (VR) \citep{shamir2015stochastic,shamir2016fast,kim2019stochastic}, and stochastic VR with momentum \citep{xu2018accelerated,kim2019stochastic} power iterations have been developed, drawing a parallel literature to gradient descent for convex optimization. 

A general form of power iteration has been used to solve 
\begin{equation} \label{prob:main}
\textrm{maximize} \quad f(x) \quad \textrm{subject to} \quad x \in \partial\mathcal{B}_d \triangleq \{x \in \mathbb{R}^d : \| x \| = 1 \}
\end{equation}
in many applications such as sparse principal component analysis (PCA) \citep{journee2010generalized, luss2013conditional}, $L_1$-norm kernel PCA \citep{kim2019simple}, phase synchronization \citep{liu2017estimation}, and the Burer-Monteiro factorization of semi-definite programs \citep{erdogdu2018convergence}. (All norms are 2-norms unless indicated otherwise.) Nevertheless, theoretical understanding of when such algorithms enjoy the attractive convergence property of power iteration is limited. Only global sublinear convergence has been shown for convex $f$ \citep{journee2010generalized}, not generalizing the appealing linear convergence property of power iteration.

In view of manifold optimization \citep{absil2009optimization}, scale invariant problems \eqref{prob:main} can be seen as an optimization problem on the real projective plane.
Through reformulations, one can obtain an unconstrained optimization problem on the embedding space, which can be solved by general non-convex optimization algorithms such as gradient-based methods with line search or trust region methods. However, these algorithms require hyperparameters such as the step size while power iteration does not.

In this work, we introduce a new class of optimization problems called \textit{scale invariant problems} and show that they can be efficiently solved by a general form of power iteration called \textit{scale invariant power iteration} (SCI-PI) with a generalized convergence guarantee of power iteration. 
We say that an optimization problem is a scale invariant problem if the objective function $f$ is \textit{scale invariant} in \eqref{prob:main}. A function $f$ is called scale invariant, which is rigorously defined later, if its geometric surface is invariant under constant multiplication of $x$. Many important optimization problems in statistics and machine learning can be formulated as scale invariant problems, for instance, $L_p$-norm kernel PCA and maximum likelihood estimation of mixture proportions, to name a few. Moreover, as studied herein, independent component analysis (ICA), non-negative matrix factorization (NMF), and Gaussian mixture models (GMM) can be formulated as extended settings of scale invariant problems.

Derivatives of scale invariant functions have the interesting relation that $\nabla^2 f(x) x = k \nabla f(x)$ holds for some $k$. Using the KKT condition, we derive an eigenvector property stating that any stationary point $x^*$ satisfying $\nabla f(x^*) = \lambda^* x^*$ for some $\lambda^*$ is an eigenvector of $\nabla^2 f(x^*)$. Due to the eigenvector property, scale invariant problems can be locally seen as the leading eigenvector problem. Therefore, we can expect that a simple update rule like power iteration would efficiently solve scale invariant problems near a local optimum $x^*$. Another interesting property of scale invariant problems is that by swapping the objective function and the constraint, a geometrically interpretable dual problem with the goal of finding the closest point $w$ to the origin from the constraint $f(w)=1$ is obtained. By mapping an iterate $x_k$ to the dual space, taking a descent step in the dual space and mapping it back to the original space, we geometrically derive SCI-PI, which replaces $Ax_k$ with $\nabla f(x_k)$ in power iteration. We show that SCI-PI converges to a local maximum $x^*$ at a linear rate when initialized close to it. The convergence rate is proportional to $\bar{\lambda}_2$ / $\lambda^*$ where $\bar{\lambda}_2$ is the spectral norm of $\nabla^2 f(x^*) (I-x^* (x^*)^T)$ and $\lambda^*$ is the Lagrange multiplier corresponding to $x^*$, generalizing the convergence rate of power iteration. Moreover, under some mild conditions, we provide an explicit expression regarding the initial condition on $\|x_0 - x^* \|$ to ensure convergence.

In the extended settings, we discuss three variants of \eqref{prob:main}. In the first setting, we consider a sum of scale invariant functions as an objective function. This setting covers a Kurtosis-based ICA and can be solved by SCI-PI with similar convergence guarantees. Second, we consider a block version of scale invariant problems which covers NMF and the Burer-Monteiro factorization of semi-definite programs. To solve this block scale invariant problem, we present a block version of SCI-PI and show that it attains linear convergence in a two-block case. Lastly, we consider partially scale invariant problems which include general mixture problems such as GMM. For this partially scale invariant problems, we present an alternative algorithm based on SCI-PI and gradient ascent along with its convergence analysis. In numerical experiments, we benchmark the proposed algorithms against state-of-the-art methods for KL-NMF, GMM and ICA. The experimental results show that our algorithms are computationally competitive and result in better solutions in ``most'' if we do not beat in all herein studied cases.

Our work has the following contributions.
\begin{enumerate}
    \item We introduce scale invariant problems which cover interesting examples in statistics and machine learning yet can be efficiently solved by a general form of power iteration due to the eigenvector property.
    \item We present a geometric derivation of SCI-PI and provide a convergence analysis for it. We show that SCI-PI converges to a local maximum $x^*$ at a linear rate when initialized close to $x^*$. This generalizes the attractive convergence property of power iteration. Moreover, we introduce three extended settings of scale invariant problems along with solution algorithms and their convergence analyses.
    \item We report numerical experiments including a novel reformulation of KL-NMF to a block scale invariant problem. The experimental results demonstrate that SCI-PI is not only computationally competitive to state-of-the-art methods but also often yield better solutions.
\end{enumerate}

The paper is organized as follows. In Section~\ref{sec:optimization-problem}, we define scale invariance and present interesting properties of scale invariant problems including an eigenvector property and a dual formulation. We then provide a geometric derivation of SCI-PI and a convergence analysis in Section~\ref{sec:generalized-power-method}. The extended settings are discussed in Section~\ref{sec:extended-settings} and we report the numerical experiments in Section~\ref{sec:numerical-experiments}.

\section{Scale Invariant Problems}
\label{sec:optimization-problem}
Before presenting properties of scale invariant problems, we first define scale invariant functions.
\begin{definition}
We say that a function $f:\mathbb{R}^{d} \rightarrow \mathbb{R}$ is multiplicatively scale invariant if it satisfies
\begin{equation}
f(cx) = u(c)f(x) \label{def:multiplicative}
\end{equation}
for some even function $u:\mathbb{R} \rightarrow \mathbb{R}^{+}$ with $u(0) = 0$. Also, we say that $f:\mathbb{R}^{d} \setminus \{0\} \rightarrow \mathbb{R}$ is additively scale invariant if it satisfies 
\begin{equation}
f(cx) = f(x) + v(c) \label{def:additive}
\end{equation}
for some even function $v:\mathbb{R} \setminus \{0\} \rightarrow \mathbb{R}$ with $v(1) = 0$.
\end{definition}
The following proposition characterizes the exact form of $u$ and $v$ for continuous $f$.
\begin{proposition} \label{prop:homogeneity} If a continuous function $f \neq 0$ satisfies \eqref{def:multiplicative} with a multiplicative factor $u$, then we have 
\begin{equation}
u(c)=|c|^p \label{eq:multiplicative-factor}
\end{equation}
for some $p > 0$. Also, if a continuous function $f$ satisfies \eqref{def:additive} with an additive factor $v$, then we have 
\begin{equation}
v(c)=\log_{a} |c| \label{eq:additive-factor}
\end{equation}
for some $a$ such that $0<a$ and $a \neq 1$.
\end{proposition}

\begin{proof}
We first consider the multiplicative scale invariant case. Let $x$ be a point such that $f(x) \neq 0$. Then, we have
\begin{align*}
f(rsx) = u(rs) f(x) = u(r)u(s) f(x),
\end{align*}
which results in
\begin{align*}
u(rs) = u(r)u(s)
\end{align*}
for all $r,s \in \mathbb{R}$. Let $g(r) = \text{ln}(u(e^r))$. Then, we have
\begin{align*}
g(r+s) = \text{ln}(u(e^{r+s})) = \text{ln}(u(e^re^s)) = \text{ln}(u(e^r)) +\text{ln}(u(e^s)) = g(r) + g(s),
\end{align*}
which implies that $g$ satisfies the first Cauchy functional equation. Since $f$ is continuous, so is $u$ and thus $g$. Therefore, by \citep[pp.~81-82]{sahoo2011introduction}, we have
\begin{align}
g(r) = r g(1)
\label{eq:Cauchy-multiplicative-solution}
\end{align}
for all $r \geq 0$. From the definition of $g$ and \eqref{eq:Cauchy-multiplicative-solution}, we have
\begin{align}
u(e^r) = e^{g(r)} = (e^r)^{g(1)}.
\label{eq:Cauchy-multiplicative-solution-2}
\end{align}
Representing $r > 0$ as $r=e^{\text{ln}(r)}$ and using \eqref{eq:Cauchy-multiplicative-solution-2}, we obtain
\begin{align*}
u(r) = u \left( e^{\text{ln}(r)} \right) = r^{g(1)} = r^{\text{ln} \left( u(e) \right)} = r^p.
\end{align*}
Since $f(x) \neq 0$, if $p = \text{ln}(u(e)) < 0$, then we have
$$
\text{lim}_{r \rightarrow 0_+} f(r{x}) = \text{lim}_{r \rightarrow 0_+} u(r) f({x}) = f({x}) \cdot \text{lim}_{r \rightarrow 0_+} r^p = f(x) \cdot \infty \neq f(0) < \infty,
$$
contradicting the fact that $f$ is continuous at $0$. Also, if $p=0$, then we get $u(r)=1$, which contradicts $u(0)=0$. Therefore, we must have $p>0$. From $u$ being an even function, we finally have
\begin{align*}
u(r) = |r|^p
\end{align*}
for $r \in \mathbb{R}$.

\vspace{2mm}

Now, consider the additive scale invariant case. For any $x \in \text{dom}(f)$, we have
\begin{align*}
f(rsx) = f(x) + v(rs) = f(x) + v(r) + v(s),
\end{align*}
which results in
\begin{align*}
v(rs) = v(r) + v(s)
\end{align*}
for all $r,s \in \mathbb{R}$. Let $g(r) = v(e^r)$. Then, we have
\begin{align*}
g(r+s) = v(e^{r+s}) = v(e^re^s) = v(e^r) + v(e^s) = g(r) + g(s).
\end{align*}
Since $g$ is continuous and satisfies the second Cauchy functional equation, by
\citep[pp.~83-84]{sahoo2011introduction}, we have
\begin{align*}
g(r) = r g(1)
\end{align*}
for all $r \geq 0$. For $r>0$, letting $r = e^{\text{ln}(r)}$, we have
\begin{align*}
v(r) = v(e^{\text{ln}(r)}) = g(\text{ln}(r)) = g(1) \text{ln}(r) = v(e) \text{ln}(r) = \text{log}_a (r)
\end{align*}
where $a = e^{\frac{1}{v(e)}}$. Note that $a$ satisfies $0<a$ and $a \neq 1$. From the fact that $v$ is an even function, we finally have
\begin{align*}
v(r) = \text{log}_a |r|
\end{align*}
for $r \in \mathbb{R} \setminus \{0\}$.
\end{proof}

Using the explicit forms of $u$ and $v$ in Proposition~\ref{prop:homogeneity}, we establish derivative-based properties of scale invariant functions below.

\begin{proposition} \label{prop:scale-inv}
Suppose that $f$ is twice differentiable. If $f$ satisfies \eqref{def:multiplicative} with a multiplicative factor $u(c)=|c|^p$, we have
\baa
c \nabla f(cx) = |c|^p \nabla f(x), \quad \nabla f(x)^T x = p f(x), \quad \nabla^2 f(x)x = (p-1)\nabla f(x).
\eaa
Also, if $f$ satisfies \eqref{def:additive} with an additive factor $v(c) = \log_a |c|$, we have
\baa
c\nabla f(cx) = \nabla f(x), \quad \nabla f(x)^T x = \log^{-1}(a),\quad \nabla^2 f(x)x = -\nabla f(x).
\eaa
\end{proposition}

\begin{proof}
Without loss of generality, we can represent a scale-invariant function $f$ as 
\begin{align}
f(cx) = u(c)f(x) + v(c) \label{eq:scale-inv-general}
\end{align}
since we can restore a multiplicatively or additively scale-invariant function by setting $v(c)=0$ or $u(c)=1$, respectively. By differentiating \eqref{eq:scale-inv-general} with respect to $x$, we have
\begin{equation*}
\nabla f(cx) = \frac{u(c)}{c} \nabla f(x).
\end{equation*}
On the other hand, by differentiating \eqref{eq:scale-inv-general} with respect to $c$, we have
\begin{align}
\nabla f(cx)^T x = u'(c) f(x) + v'(c). \label{eq:scale-inv-general-diff-c}
\end{align}
By differentiating \eqref{eq:scale-inv-general-diff-c} with respect to $x$, we obtain
\begin{align}
c \nabla^2 f(cx) x + \nabla f(cx) = u'(c) \nabla f(x). \label{eq:scale-inv-general-diff-c-x}
\end{align}
Plugging $c=1$ into \eqref{eq:scale-inv-general-diff-c} and \eqref{eq:scale-inv-general-diff-c-x} completes the proof.
\end{proof}

Proposition~\ref{prop:scale-inv} states that a scale invariant function satisfies $\nabla^2 f(x) = k\nabla f(x)$ holds for some $k$. This relation is interesting since using the first-order optimality conditions, we can derive an eigenvector property as follows.

\begin{proposition}
\label{prop:eigenvalue}
Suppose that $f$ is twice differentiable and let $(\lambda^*, x^*)$ be a stationary point of \eqref{prob:main} such that 
\begin{equation*}
\nabla f(x^*) = \lambda^* x^*.
\end{equation*}
If $f$ satisfies \eqref{def:multiplicative} with $u(c)=|c|^p$, then we have 
\begin{equation*}
\nabla^2 f(x^*) x^* = (p-1) \lambda^* x^*.
\end{equation*}
Also, if $f$ satisfies \eqref{def:additive} with $v(c)=\log_a |c|$, then we have 
\begin{equation*}
\nabla^2 f(x^*) x^* = - \lambda^* x^*.    
\end{equation*}
In both cases, $x^*$ is an eigenvector of $\nabla^2 f(x^*)$. Moreover, if $\lambda^*$ is greater than the largest eigenvalue of $\nabla^2 f(x^*)(I-x^* (x^*)^T)$, then $x^*$ is a local maximum to \eqref{prob:main}.
\end{proposition}

\begin{proof}
Consider the Lagrangian function
\begin{align*}
L(x,\lambda) = f(x) + \frac{\lambda}{2} \left( 1-\|x\|^2 \right)
\end{align*}
and a stationary point $(\lambda^*, x^*)$ satisfying
\begin{align*}
\nabla f(x^*) = \lambda^* x^*, \quad \|x^*\| = 1.
\end{align*}
If $f$ is multiplicative scale invariant with the degree of $p$, by Proposition~\ref{prop:scale-inv}, we have
\begin{align*}
\nabla^2 f(x^*)x^* = (p-1) \nabla f(x^*) = (p-1) \lambda^* x^*.
\end{align*}
Also, by Proposition~\ref{prop:scale-inv}, if $f$ is additive scale invariant $f$, we have
\begin{align*}
\nabla^2 f(x^*)x^* = - \nabla f(x^*) = - \lambda^* x^*.
\end{align*}
Therefore, in both cases, a stationary point $x^*$ is an eigenvector of $\nabla^2 f(x^*)$.

Suppose that $\lambda^*$ is greater than the largest eigenvalue of $\nabla^2 f(x^*)(I-x^*(x^*)^T)$. 
For any $d$ satisfying $d^Tx^* = 0$, we have
\begin{equation*}
d^T \nabla_{xx}^2 L(x^*,\lambda^*) d = d^T \nabla^2 f(x^*)(I-x^*(x^*)^T) d - \lambda^* \| d\|^2 < 0.
\end{equation*}
Since the second-order sufficient condition is satisfied, $x^*$ is a local maximum.
\end{proof}

Proposition~\ref{prop:eigenvalue} states that a stationary point $x^*$ is an eigenvector of $\nabla^2 f(x^*)$. Note that the Lagrange multiplier $\lambda^*$ is not necessarily an eigenvalue corresponding to $x^*$. The eigenvalue corresponding to $x^*$ is $(p-1) \lambda^*$ if $f$ is multiplicatively scale invariant or $- \lambda^*$ if $f$ is additively scale invariant.
The sufficient condition for local optimality requires that the Lagrange multiplier $\lambda^*$ rather than the eigenvalue corresponding to $x^*$ is greater than the largest eigenvalue of $\nabla^2 f(x^*)(I-x^*(x^*)^T)$. Due to this eigenvector property, scale invariant problems can be considered as a generalization of the leading eigenvector problem. Next, we introduce a dual formulation of scale invariant problems.

\begin{proposition} \label{thm:reformulation}
Suppose that a continuous function $f$ is either multiplicatively scale invariant such that $f(x^*) > 0$ or additively scale invariant with an additive factor $u(c) = \log_a |c|$ with $a>1$. Then, solving \eqref{prob:main} is equivalent to solving the following optimization problem
\begin{equation} \label{prob:dual}
\textrm{\rm minimize} \quad \|w\| \quad \textrm{\rm subject to} \quad f(w) = 1.
\end{equation}
In other words, if $x^*$ is an optimal solution to \eqref{prob:main}, then $w^* = x^*/f(x^*)^{1/p}$ (multiplicative) or $w^* = a^{1-f(x^*)} x^*$ (additive) is an optimal solution to \eqref{prob:dual}. Conversely, if $w^*$ is an optimal solution to \eqref{prob:dual}, $x^*=w^*/\|w^*\|$ is an optimal solution to \eqref{prob:main}.
\end{proposition}

\begin{proof}
First, we consider the case where an objective function $f$ is multiplicative scale invariant with a multiplicative factor $u(c)=|c|^p$ where $p > 0$. Let $w^*$ be an optimal solution to \eqref{prob:dual}. From that $f(w^*)=1$, we have $w^* \neq 0$, which leads to 
$\|w^*\|>0$ and $f \left( {w^*}/{\|w^*\|} \right) = {1}/{\|w^*\|^p} > 0$.
Suppose an optimal solution to \eqref{prob:main} is $y$ with 
\begin{align}
\label{proof:reformulation-1}
f(y) > f \left( {w^*}/{\|w^*\|} \right) > 0.
\end{align}
Let $\hat{y} = {y}/{f(y)^{1/{p}}}$. Then, we have $f(\hat{y}) = 1$ and ${y} = {\hat{y}}/{\|\hat{y}\|}$.
Using $f(\hat{y}) = f(w^*) = 1$, we have
\begin{align}
\label{proof:reformulation-2}
f(y) = f\left( \frac{\hat{y}}{\|\hat{y}\|} \right) = \frac{1}{\|{\hat{y}}\|^{1/p}}, \quad f \left( \frac{{w}^*}{\|{w}^*\|} \right) =  \frac{1}{\|{w}^*\|^{1/p}}.
\end{align}
From \eqref{proof:reformulation-1} and \eqref{proof:reformulation-2}, we obtain $\|\hat{y}\| < \|{w}^*\|$, which contradicts that ${w}^{*}$ is an optimal solution to \eqref{prob:dual}.

On the other hand, let $x^*$ be an optimal solution to \eqref{prob:main} with $f(x^*)>0$. Suppose that an optimal solution to \eqref{prob:dual} is $z$ with
\begin{align}
\label{proof:reformulation-3}
\| z \| < \| {x^*} / {f(x^*)^{1/p}} \|.
\end{align}
Let $\hat{z} = {z}/{\|z\|}$. Then, we have 
$\| \hat{z} \| = 1$ and $z = {\hat{z}}/{f(\hat{z})^{1/p}}$.
From that $\|\hat{z}\| = \|x^*\| =1$, we have
\begin{align}
\label{proof:reformulation-4}
\| z \| = \| {\hat{z}}/{f(\hat{z})^{1/p}} \| = {1}/{f(\hat{z})^{1/p}}, \quad \| {x^*}/{f(x^*)^{1/p}} \|  = {1}/{f(x^*)^{1/p}}.
\end{align}
From \eqref{proof:reformulation-3} and \eqref{proof:reformulation-4}, we have
\begin{align*}
f(x^*) < f(\hat{z})
\end{align*}
since $p>0$, which contradicts the assumption that $x^*$ is an optimal solution to \eqref{prob:main}.

Next, let $f$ be an additively scale invariant function with an additive factor $v(c)=\text{log}_a |c|$ with $a>1$. In the same way as above, let $w^*$ be an optimal solution to \eqref{prob:dual} and suppose that an optimal solution of \eqref{prob:main} is $y$ with
\begin{align}
\label{proof:reformulation-5}
f(y) > f\left( {w^*}/{\|w^*\|} \right).  
\end{align}
Let $\hat{y} = a^{1-f(y)} y$. Then, we have $f(\hat{y}) = 1$ an ${y} = {\hat{y}}/{\|\hat{y}\|}$. 
Since $f(\hat{y})=f(w^*)=1$, we have
\begin{align}
\label{proof:reformulation-6}
f(y) = f(\hat{y}) - \text{log}_a \|\hat{y}\| = 1 - \text{log}_a \|\hat{y}\|, \quad f \left({{w}^*}/{\|{w}^*\|} \right) = 1 - \text{log}_a \|{w}^*\|.
\end{align}
From \eqref{proof:reformulation-5} and \eqref{proof:reformulation-6}, we have 
\begin{align*}
\|\hat{y}\| < \|{w}^*\|
\end{align*}
due to $a>1$, contradicting the fact that ${w}^{*}$ is an optimal solution to \eqref{prob:dual}. 

Conversely, let $x^*$ be an optimal solution to \eqref{prob:main} and suppose that an optimal solution to \eqref{prob:dual} is $z$ with
\begin{align}
\label{proof:reformulation-7}
\| z \| < \| a^{1-f(x^*)} x^* \|.
\end{align}
Let $\hat{z} = {z}/{\|z\|}$. Then, we have $\| \hat{z} \| = 1$ and $z = a^{1-f(\hat{z})} \hat{z}$. Using $\|\hat{z}\| = \|x^*\| = 1$, we have
\begin{align}
\label{proof:reformulation-8}
\| z \|  = a^{1-f(\hat{z})}, \quad \| a^{1-f(x^*)} x^* \|  = a^{1-f(x^*)}.
\end{align}
From \eqref{proof:reformulation-7} and \eqref{proof:reformulation-8}, we have
\begin{align*}
f(x^*) < f(\hat{z})
\end{align*}
due to $a>1$, contradicting the assumption that $x^*$ is an optimal solution to \eqref{prob:main}.
\end{proof}

Note that a dual reformulation for a multiplicatively scale invariant $f$ with $f(x^*) < 0$ or an additively scale invariant $f$ with $0<a<1$ can be obtained by replacing $f(w) = 1$ with $f(w) = -1$ in \eqref{prob:dual}. The dual formulation \eqref{prob:dual} has a nice geometric interpretation that an optimal solution $w^*$ is the closest point to the origin from $\{w:f(w)=1\}$. We use this understanding to derive SCI-PI in Section~\ref{sec:generalized-power-method}.


Lastly, we introduce two well-known examples of scale invariant problems in machine learning and statistics.
\begin{example}[$L_p$-norm Kernel PCA] \label{eg:pca}
Given data vectors $a_i \in \mathbb{R}^d$ and a mapping $\Phi$, $L_p$-norm PCA considers
\baa \label{eqn:lp-norm-pca}
\textrm{\rm maximize} \quad \frac{1}{n} {\textstyle \sum_{i=1}^n} \lVert{\Phi(a_i)^T x}\rVert_p^p \quad \textrm{\rm subject to} \quad x \in \partial \mathcal{B}_{d}
\eaa
where the objective function satisfies property \eqref{def:multiplicative} with $u(c) = |c|^p$. 
\end{example}
\begin{example}[Estimation of Mixture Proportions]
\label{eg:mixture}
Given a design matrix $L \in \mathbb{R}^{n \times d}$ satisfying $L_{jk} \geq 0$, the problem of estimating mixture proportions seeks to find a vector $\pi$ of mixture proportions on the probability simplex  $\mathcal{S}^d =
  \big\{ \pi : {\textstyle \sum_{k=1}^d \pi_k = 1},\ \pi \geq 0\big\}$ that maximizes the log-likelihood ${\textstyle \sum_{j=1}^n} \log \left({\textstyle \sum_{k=1}^d} L_{jk} \pi_k \right)$. 
By reparametrizing $\pi_k$ by $x_k^2$, we obtain an equivalent optimization problem
\baa \label{eqn:mixture-reform}
\textrm{\rm maximize}\quad \frac{1}{n} {\textstyle \sum_{j=1}^n} \log
                        \left({\textstyle \sum_{k=1}^d}L_{jk}x_k^2 \right) \quad \textrm{\rm subject to} \quad x \in \partial \mathcal{B}_{d},
\eaa
which now satisfies property \eqref{def:additive} with $v(c) = 2\log |c|$.
\end{example}
The reformulation idea in Example~\ref{eg:mixture} implies that any simplex-constrained problem with scale invariant $f$ can be reformulated to a scale invariant problem.

\section{Scale Invariant Power Iteration}
\label{sec:generalized-power-method}
In this section, we provide a geometric derivation of SCI-PI to find a local optimal solution of \eqref{prob:main}. The algorithm is developed using the geometric interpretation of the dual formulation \eqref{prob:dual} as illustrated in Figure \ref{fig:idea}. Starting with an iterate $x_k \in \partial \mathcal{B}$, we obtain a dual iterate $w_k$ by projecting $x_k$ to the constraint $f(w)=1$. Given $w_k$, we identify the hyperplane $h_{k}$ which the current iterate $w_{k}$ lies on and is tangent to $f(w)=1$. After identifying the equation of $h_{k}$, we find the closest point $z_k$ to the origin from $h_{k}$ and obtain a new dual iterate $w_{k+1}$ by projecting $z_k$ to the constraint $f(w)=1$. Finally, we obtain a new primal iterate $x_{k+1}$ by mapping $w_{k+1}$ back to the set $\partial \mathcal{B}_d$.

Now, we develop an algorithm based on the above idea. For derivation of the algorithm, we assume that an objective function $f$ is continuous and satisfies either \eqref{def:multiplicative} with $u(c)=|c|^p$ where $p>0$ and $f(x)>0$ for all $x \in \partial \mathcal{B}$ or \eqref{def:additive} with $v(c)=\text{log}_a |c|$ where $1<a$. Under these conditions, a scalar mapping from $x_k$ to $w_k$ can be well defined as $w_k = x_k/f(x_k)^{1/p}$ or $w_k = a^{1-f(x_k)} x_k$, respectively. Let $w_k = c_k x_k$. Since $w_k$ is on the constraint $f(w)=1$, the tangent vector of the hyperplane $h_k$ is $\nabla f(w_k)$. Therefore, we can write down the equation of the hyperplane $h_k$ as
$\left \{ w: \nabla f(w_k)^T (w-w_k)=0 \right \}$.
Note that $z_k$ is a scalar multiple of $\nabla f(w_k)$ where the scalar can be determined from the requirement that $z_k$ is on $h_k$. Since $w_{k+1}$ is the projection of $z_k$, it must be a scalar multiple of the tangent vector $y_k = \nabla f(w_k)$. Therefore, we can write $w_{k+1}$ as $w_{k+1}=d_k y_k$. Finally, by projecting $w_{k+1}$ to $\partial \mathcal{B}$, we obtain
\begin{align*}
x_{k+1} = \frac{w_{k+1}}{\|w_{k+1}\|} = \frac{d_{k} y_{k}}{\|d_{k} y_{k}\|} = \frac{y_{k}}{\|y_{k}\|} = \frac{\nabla f(w_k)}{\| \nabla f(w_k) \|} = \frac{\nabla f(c_k x_k)}{\| \nabla f(c_k x_k) \|} = \frac{\nabla f(x_k)}{\| \nabla f(x_k) \|}
\end{align*}
where the last equality follows from Proposition~\ref{prop:scale-inv}. Summarizing all the above, we obtain SCI-PI presented in Algorithm \ref{alg:SCI-PI}.

\begin{minipage}{0.45 \textwidth}
\centering
\vspace{0.1in}
\begin{tikzpicture}[scale=1, every node/.style={scale=0.8}]
\filldraw (8-7+1, 0) circle [radius=1pt] node[font=\fontsize{10}{10}, above=0.05] {$(0,0)$};
\filldraw (1+1.4142+1, 0-1.4142) circle [radius=1pt] node[font=\fontsize{10}{10}, right=0.1] {$x_k$};
\filldraw (8+0.8944-7+1, -0.8944) circle [radius=1pt] node[font=\fontsize{10}{10}, above=0.2] at (8+0.8944+0.1-7+1,-0.8944-0.2) {$w_k$};
\draw (6-7+1, -1.6180) node[font=\fontsize{10}{10}, left=0.1] {$h_k$} -- (11-7+1,-0.3680);
\filldraw (1+0.4851+1, 0-1.9403) circle [radius=1pt] node[font=\fontsize{10}{10}, below=0.05] {$x_{k+1}$};
\filldraw (8+0.2481-7+1, 0-0.9923) circle [radius=1pt] node[font=\fontsize{10}{10}] at (8+0.2481-0.4-7+1,0-0.9923+0.15) {$w_{k+1}$};
\filldraw (8+0.2631-7+1, 0-1.0523) circle [radius=1pt] node[font=\fontsize{10}{10}] at (8+0.2631+0.2-7+1,0-1.0523-0.2) {$z_{k}$};
\draw (1+1,0) circle [radius = 2] node[font=\fontsize{10}{10}, above=2.0] {$\partial \mathcal{B}_d$};
\draw (8-7+1,0) ellipse [x radius = 2, y radius = 1] node[font=\fontsize{10}{10}, above=1] {$f(w) = 1$}; \draw [->,scale=1] (1+1.4142+1, 0-1.4142) to (8+0.8944-7+1, -0.8944);
\draw [->,scale=1] (8+0.2481-7+1, 0-0.9923) to (1+0.4851+1, 0-1.9403);
\draw (8+0.2631-7+1, 0-1.0523) -- (8-7+1, 0);
\draw (8+0.2-7+1,-0.8) -- (8+0.4043-7+1, -0.7489) -- (8+0.2631+0.4043-0.2-7+1, 0-1.0523+0.8-0.7489);
\end{tikzpicture}

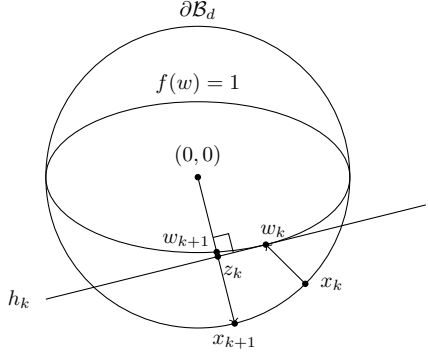
\captionof{figure}{Geometric derivation of SCI-PI}
\label{fig:idea}
\vspace{0.1in}
\end{minipage}
\hfill 
\begin{minipage}{0.45 \textwidth}
\vspace{0.1in}
\begin{algorithm}[H]
\caption{SCI-PI}
\label{alg:SCI-PI}
\begin{algorithmic}
   \STATE {\bfseries Input:} initial point $x_0$ \\
   \FOR{$k=0,1,\ldots,T-1$}
   \STATE $x_{k+1} \leftarrow \dfrac{\nabla f(x_k)}{\| \nabla f(x_k) \|}$ \\
   \ENDFOR
   \STATE {\bfseries Output:} $x_{T}$
\end{algorithmic}
\end{algorithm}
\vspace{0.1in}
\end{minipage}

Next, we provide a convergence analysis of SCI-PI. 

Global sublinear convergence of SCI-PI for convex $f$ has been addressed in \cite{journee2010generalized}. We additionally show that SCI-PI yields an ascent step even for quasi-convex $f$.
\begin{proposition} \label{prop:ascent-quasi-convex}
If $f$ is quasi-convex and differentiable, a sequence of iterates $\{x_k\}_{k=0,1,\cdots}$ generated by SCI-PI satisfies
$f(x_{k+1}) \geq f(x_{k})$ for $k = 0,1,\cdots$.
\end{proposition}

\begin{proof}
If $f(x_{k+1}) < f(x_k)$, by the first-order condition of differentiable quasi-convex functions, we have
\begin{align}
\label{proof:first-order-quasi-convex}
\nabla f(x_k)^T(x_{k+1}-x_k) = \nabla f(x_k)^T \bigg(\frac{\nabla f (x_k)}{\| \nabla f(x_k) \|} - x_k \bigg) = \| \nabla f(x_k) \| - \nabla f(x_k)^T x_k \leq 0.
\end{align}
However, since $f(x_{k+1}) \neq f(x_k)$, $\nabla f(x_k)$ is not a scalar multiple of $x_k$, leading to
\begin{equation*}
\| \nabla f(x_k) \| - \nabla f(x_k)^T x_k > 0.
\end{equation*}
This contradicts \eqref{proof:first-order-quasi-convex}. Therefore, we should have $f(x_{k+1}) \geq f(x_k)$.
\end{proof}

If $f$ is quasi-convex, the set $\{ w : f(w) \leq 1 \}$ is convex, therefore, from Figure~\ref{fig:idea}, we can expect that SCI-PI would yield an ascent step. If $f$ is not quasi-convex, $\{ f(x_k) \}_{k=0,1,\cdots}$ is not necessarily increasing, making it hard to analyze global convergence. Assuming that an initial point $x_0$ is close to a local maximum $x^*$, we study local convergence of SCI-PI as follows.

\begin{theorem} \label{thm:GPM-local-convergence}
Let $f$ be a scale invariant, twice continuously differentiable function on an open set containing $\partial \mathcal{B}_d$ and let $x^*$ be a local maximum satisfying $\nabla f(x^*) = {\lambda^*} x^*$ and $\lambda^* > \bar{\lambda}_2 = {\textstyle \max_{2 \leq i \leq d}} |\lambda_i|$
where $(\lambda_i,v_i)$ is an eigen-pair of $\nabla^2 f(x^*)$ with $x^* = v_1$. Then, there exists some $\delta > 0$ such that under the initial condition $1 - x_0^Tx^* < \delta$, the sequence of iterates $\{x_k\}_{k=0,1,\cdots}$ generated by SCI-PI satisfies
\begin{align*}
1 - (x_k^Tx^*)^2
\leq 
\prod_{t=0}^{k-1} \bigg( \frac{\bar{\lambda}_2}{\lambda^*} + \gamma_t \bigg)^{2} \left( 1 - (x_0^Tx^*)^2 \right), \, \,
\end{align*}
where
\begin{align*}
\frac{\bar{\lambda}_2}{\lambda^*} + \gamma_t < 1  \, \, \textup{for all } t \geq 0 \, \, \textup{and} \, \, \lim_{k \rightarrow \infty} \gamma_k = 0.
\end{align*}
Moreover, if $\nabla_i f = {\partial f}/{\partial x_i}$ has a continuous Hessian $H_i$ on an open set containing $\mathcal{B}_{d,\infty} \triangleq \{ x \in \mathbb{R}^d : \| x \|_{\infty} \leq 1 \}$, we can explicitly write $\delta$ as
\begin{equation*}
\delta(\lambda^*, \bar{\lambda}_1, \bar{\lambda}_2, M) = 
\min \left \{ 
\left( \frac{\lambda^*}{\bar{\lambda}_1+M} \right)^2 , \, 
\left( 
\frac{\lambda^*-\bar{\lambda}_2}{\bar{\lambda}_1 + 2M} \right)^2, 1
\right \}
\end{equation*}
where $\bar{\lambda}_1 = |\lambda_1|$ and
\begin{equation*}
M = \hspace{-2mm} \max_{x \in \partial \mathcal{B}_d, \\ y^1, \cdots, y^d \in \mathcal{B}_{d,\infty}} \hspace{-2mm} \sqrt{\textstyle \sum_{i=1}^d ( x^T G_i(y^1,\cdots,y^d) x )^2}, \quad G_i(y^1,\cdots,y^d) = {\textstyle \sum_{j=1}^d v_{i,j} H_j(y^j)}.
\end{equation*}
\end{theorem}

\begin{proof}
Since $\nabla^2 f(x^*)$ is real and symmetric, without loss of generality, we assume that $\{ v_1, \ldots, v_d\}$ form an orthogonal basis in $\mathbb{R}^d$. 

Since $f$ is twice continuously differentiable on an open set containing $\partial \mathcal{B}_d$, for $x \in \partial \mathcal{B}_d$, using the Taylor expansion of $\nabla f(x)^Tv_i$ at $x^*$, we have
\begin{align}
\label{proof:SCI-PI-Taylor}
\nabla f(x)^Tv_i = \nabla f(x^*)^Tv_i + (x - x^*)^T\nabla^2 f(x^*) v_i + R_i(x)
\end{align}
where 
\begin{align}
\label{proof:SCI-PI-Taylor-remainder}
R_i(x) = o(\|x-x^*\|).
\end{align}
From $\nabla f(x^*) = \lambda^* x^*$ and $x^* = v_1$, we have
\begin{equation}
\label{proof:grad-x-product-v1}
\begin{aligned}
\nabla f(x)^T v_1 &= \nabla f(x^*)^Tx^* + (x - x^*)^T \nabla^2 f(x^*) x^*  + R_1(x) \\
&= \lambda^* - \lambda_1(1 - x^Tx^*)  + R_1(x) \\
&= \lambda^* + \alpha(x)
\end{aligned}
\end{equation}
where 
\begin{equation*}
\alpha(x) = - \lambda_1(1 - x^Tx^*) + R_1(x) = o(\|x-x^*\|)
\end{equation*}
due to $R_1(x) = o(\|x-x^*\|)$ and $1 - x^Tx^* = o(\|x-x^*\|)$.

On the other hand, for $2 \leq i \leq d$, due to $\nabla f(x^*) = \lambda^* x^*$, we have 
\begin{align}
\label{proof:grad-x-star-product-vi}
\nabla f(x^*)^T v_i = \lambda^* (x^*)^T v_i = 0.
\end{align}
From \eqref{proof:SCI-PI-Taylor}, this results in
\begin{align}
\label{proof:grad-x-product-vi}
\nabla f(x)^T v_i = \lambda_i x^Tv_i + R_i(x).
\end{align}
Let $\bar{R}_2 (x) = {\textstyle \max_{2 \leq i \leq d} |R_i(x)|}$. 
Note that $\bar{R}_2 (x) = o(\|x-x^*\|)$.
By \eqref{proof:grad-x-product-vi}, we obtain
\begin{equation}
\label{proof:SCI-PI-grad-x-product-vi-square-sum-R}
\begin{aligned}
\sum_{i=2}^d \left( \nabla f(x)^T v_i \right) ^2 &= \sum_{i=2}^d \left[ \lambda_i^2 (x^Tv_i)^2 + 2 \lambda_i (x^Tv_i) R_i(x) + \left( R_i(x) \right)^2 \right] \\
&\leq \bar{\lambda}_2^2 \sum_{i=2}^d (x^Tv_i)^2 + 2 \bar{\lambda}_2 \bar{R}_2 (x)  \sum_{i=2}^d  |x^Tv_i| + d \left( \bar{R}_2 (x) \right)^2.
\end{aligned}
\end{equation}
From $x \in \partial \mathcal{B}_d, \, x^* = v_1$, and the fact that $\{v_1,\ldots,v_d\}$ forms an orthogonal basis in $\mathbb{R}^d$, we have
\begin{align*}
\sum_{i=2}^d (x^Tv_i)^2 = 1 - (x^Tv_1)^2 = 1 - (x^Tx^*)^2 \leq 2(1-x^Tx^*) = \| x- x^*\|^2.
\end{align*}
Also, by the Cauchy Schwartz inequality, we have
\begin{equation*}
\sum_{i=2}^d  |x^Tv_i| \leq \sqrt{d} \sqrt{\sum_{i=2}^d (x^Tv_i)^2} \leq \sqrt{d} \| x- x^*\|.
\end{equation*}
Therefore, we obtain from \eqref{proof:SCI-PI-grad-x-product-vi-square-sum-R} that
\begin{equation}
\label{proof:grad-x-product-vi-square-sum}
\begin{aligned}
\sum_{i=2}^d \left( \nabla f(x)^T v_i \right) ^2 &\leq \bar{\lambda}_2^2 \| x- x^*\|^2 + 2 \bar{\lambda}_2 \bar{R}_2 (x) \sqrt{d} \| x- x^*\| + d \left( \bar{R}_2 (x) \right)^2 \\
&= \left( \bar{\lambda}_2 \| x- x^*\| + \beta(x) \right)^2
\end{aligned}
\end{equation}
where 
\begin{equation*}
\beta(x) = \sqrt{d} \bar{R}_2 (x) = o(\|x - x^* \|).
\end{equation*}
By \eqref{proof:grad-x-product-v1}, \eqref{proof:grad-x-product-vi-square-sum}, and Lemma~\ref{lemma:SCI-PI-local-convergence-small-delta}, we obtain the first part of the desired result.

Next, we consider the case where $\nabla_i f$ has a continuous Hessian $H_i$. From $\nabla_i f(x)$ being twice continuously differentiable in $\mathcal{B}_{\infty}$, we have
\begin{align}
\nabla_i f(x_k) = \nabla_i f(x^*) + \nabla \nabla_{i} f(x^*) (x_k - x^*) + \frac{1}{2} \left( x_k -x^* \right)^T H_i(\hat{x}_k^i) \left( x_k - x^* \right)
\label{eq:Taylor-gi}
\end{align}
where 
\begin{align*}
\hat{x}_{k}^{i} \in \mathcal{N} (x_k, x^*) \triangleq \left \{x : x_s = t_s x_s^* + ( 1-t_s) x_{k,s}, \, 0 \leq t_s \leq 1, \, s = 1, \ldots, d \right \}.
\end{align*}
In the above, $x_s^*$ and $x_{k,s}$ denote the $s^{th}$ coordinates of $x^*$ and $x_k$, respectively.

For each $1 \leq i \leq d$, we have
\begin{equation*}
\frac{1}{2}
\sum_{j=1}^d v_{i,j} \left( x_k -x^* \right)^T H_j(\hat{x}_{k}^{j}) \left( x_k - x^* \right)
= \frac{1}{2} (x_k-x^*)^T G_i(\hat{x}_{k}^{1},\cdots,\hat{x}_{k}^{d}) (x_k - x^*).
\end{equation*}
From
\begin{equation}
\begin{aligned}
\label{proof:SCI-PI-second-order-eq}
& \big | (x_k-x^*)^T G_i(\hat{x}_{k}^{1},\cdots,\hat{x}_{k}^{d}) (x_k - x^*) \big | \\
& \qquad \qquad = 
\| x_k - x^* \|^2 \bigg | \left[ \frac{x_k-x^*}{\|x_k-x^*\|} \right]^T G_i(\hat{x}_{k}^{1},\cdots,\hat{x}_{k}^{d}) \left[ \frac{x_k-x^*}{\|x_k-x^*\|} \right] \bigg|
\end{aligned}
\end{equation}
and
\begin{align*}
\max_{x \in \partial \mathcal{B}_d} \,\, |x^T G_i(\hat{x}_{k}^{1},\cdots,\hat{x}_{k}^{d}) x| 
\leq \max_{\substack{x \in \partial \mathcal{B}_d, \\ y^1, \cdots, y^d \in \mathcal{B}_{\infty}}} |x^T G_i({y}^{1},\cdots,{y}^{d}) x| \leq M,
\end{align*}
we have
\begin{align*}
\big | (x_k-x^*)^T G_i(\hat{x}_{k}^{1},\cdots,\hat{x}_{k}^{d})  (x_k - x^*) \big | \leq M \| x_k - x^* \|^2,
\end{align*}
leading to
\begin{align}
\frac{1}{2} \,
\big | \sum_{j=1}^d v_{i,j} \left( x_k -x^* \right)^T H_j(\hat{x}_{k}^{j}) \left( x_k - x^* \right) \big|
\leq \frac{1}{2} M \|x_k-x^*\|^2.
\label{proof:SCI-PI-second-order-bound}
\end{align}

From \eqref{eq:Taylor-gi}, \eqref{proof:SCI-PI-second-order-bound} and that $x^* = v_1$, we have
\begin{align*}
\begin{aligned}
\nabla f(x_k)^T v_1 &\geq \nabla f(x^*)^T x^* + (x_k - x^*)^T \nabla^2 f(x^*) x^* -  \frac{M}{2}  \| x_k - x^* \|^2,
\end{aligned}
\end{align*}
resulting in 
\begin{align}
\label{proof:product-v1-final}
\nabla f(x_k)^T v_1 &\geq  \lambda^* - (M + |\lambda_1|) (1 - x_k^Tx^*).
\end{align}

For $2 \leq i \leq d$, we have 
\begin{align} 
\nabla f(x_k)^T v_i &= \nabla f(x^*)^T v_i + (x_k - x^*)^T \nabla^2 f(x^*) v_i + \frac{1}{2} (x_k-x^*)^T G_i(\hat{x}_{k}^{1},\cdots,\hat{x}_{k}^{d}) (x_k - x^*) \nonumber \\
&= \lambda_i x_k^T v_i + \frac{1}{2} (x_k-x^*)^T G_i(\hat{x}_{k}^{1},\cdots,\hat{x}_{k}^{d}) (x_k - x^*). \label{proof:product-vi}
\end{align}
Using \eqref{proof:SCI-PI-second-order-eq} and 
\begin{align*}
\max_{x \in \partial \mathcal{B}_d} \,\, \sum_{i=2}^d (x^T G_i(\hat{x}_{k}^{1},\cdots,\hat{x}_{k}^{d}) x)^2 
\leq    
\max_{\substack{x \in \partial \mathcal{B}_d, \\ y^1, \cdots, y^d \in \mathcal{B}_{\infty}}} \sum_{i=2}^d (x^T G_i(y^{1},\cdots,y^{d}) x)^2 
\leq M,
\end{align*}
we have
\begin{align}
\label{proof:SCI-PI-second-order-sum-bound}
\sum_{i=2}^d \big[ (x_k-x^*)^T G_i(\hat{x}_{k}^{1},\cdots,\hat{x}_{k}^{d}) (x_k - x^*) \big]^2
\leq M^2 \|x_k-x^*\|^4.
\end{align}
Using \eqref{proof:product-vi}, \eqref{proof:SCI-PI-second-order-sum-bound} and the Cauchy-Schwartz inequality, we have
\begin{align}
\sum_{i=2}^d (\nabla f(x_k)^T v_i)^2 &\leq \sum_{i=2}^d \left( |\lambda_i| |x_k^Tv_i| + \frac{1}{2} (x_k-x^*)^T G_i(\hat{x}_{k}^{1},\cdots,\hat{x}_{k}^{d}) (x_k - x^*) \right)^2 \nonumber \\
&\leq \bar{\lambda}_2^2 \sum_{i=2}^d (x_k^Tv_i)^2 +  \bar{\lambda}_2 M  \| x_k - x^* \|^2 \sqrt{\sum_{i=2}^d (x_k^Tv_i)^2} + \frac{M^2}{4} \| x_k - x^* \|^4 \nonumber \\
&= \left( \bar{\lambda}_2 \sqrt{1 - (x_k^Tx^*)^2} + \frac{M}{2}  \| x_k - x^* \|^2 \right)^2.
\label{proof:rewrite2}
\end{align}
Using \eqref{proof:product-v1-final}, \eqref{proof:rewrite2}, and Lemma~\ref{lemma:SCI-PI-local-convergence-explicit-delta} with
$$
A = \lambda^*, \, B = M + |\lambda_1|, \, C = 0, \, D = \bar{\lambda}_2, \, E =0, \, F = M,
$$
we obtain the desired result.
\end{proof}

Theorem~\ref{thm:GPM-local-convergence} presents a local convergence result of SCI-PI with the rate being $\lambda^*/\bar{\lambda}_2$. For the leading eigenvector problem, this rate specializes to $\lambda_1/\lambda_2$, generalizing the convergence rate of power iteration. 
Note that Theorem~\ref{thm:GPM-local-convergence} requires that a Lagrange multiplier $\lambda^*$ corresponding to a local maximum $x^*$ satisfies $\lambda^* > \bar{\lambda}_2 = \max_{2 \leq i \leq d} |{\lambda}_i|$. This assumption is satisfied by all local maxima if $f$ is convex, multiplicatively scale invariant or concave, additively scale invariant. However, in general, not all local maxima satisfy this assumption since it is stronger than the local optimality condition stated as $\lambda^* > \max_{2 \leq i \leq d} {\lambda}_i$. Nevertheless, by adding $\sigma \|x\|^2$ for some $\sigma>0$ to the objective function $f$, we can always enforce $\lambda^* > \bar{\lambda}_2$.
Conversely, by adding $\sigma \|x\|^2$ for some $\sigma < 0$, we may improve the convergence rate as in shifted power iteration.

\section{Extended Settings}
\label{sec:extended-settings}
\subsection{Sum of Scale Invariant Functions} \label{subsec:sumofscaleinv}
Consider a sum of scale invariant functions having the form of 
$f(x) = \sum_{i=1}^m g_i(x) + \sum_{j=1}^n h_j(x)$ where $g_i$ is a multiplicatively scale invariant function with $u(c)=|c|^{p_i}$ and $h_j$ is an additively scale invariant function with $v(c) = \log_{a_j} |c|$. Note that this does not imply that $f$ is scale invariant in general. Here is an example that involves a sum of scale invariant functions.
\begin{example}[Kurtosis-based ICA]
Given a pre-processed data matrix $W \in \mathbb{R}^{n\times d}$, Kurtosis-based ICA \citep{hyvarinen2000independent} solves
\baa \label{eqn:ica}
\textrm{\rm maximize} \quad \frac{1}{n} \sum_{i=1}^n \left[(w_i^T x)^4 - 3\right]^2 \quad \textrm{\rm subject to} \quad x \in \partial \mathcal{B}_{d}.
\eaa
The objective function $f$ is a sum of scale invariant functions.
\end{example}

By Proposition~\ref{prop:scale-inv}, the gradient of $f$ has the form of
\begin{align*}
\nabla f(x) = \sum_{i=1}^{m} \nabla g_i(x) + \sum_{j=1}^{n} \nabla h_j(x) = 
F(x) x,
\end{align*}
where
\begin{equation*}
F(x) = \sum_{i=1}^{m} \left( \frac{1}{p_i-1} \right) \nabla^2 g_i(x) 
- \sum_{j=1}^{n} \nabla^2 h_j(x).
\end{equation*}
Note that a stationary point $x^*$ satisfying $\nabla f(x^*) = \lambda^* x^*$ is not necessarily an eigenvector of $\nabla^2 f(x^*)$. Instead, a stationary point $x^*$ is an eigenvector of $F(x)$. We present a local convergence analysis of SCI-PI for a sum of scale invariant functions as follows.

\begin{theorem} \label{thm:SCI-PI-local-convergence-SSI}
Let $f$ be a sum of scale invariant functions and twice continuously differentiable on an open set containing $\partial \mathcal{B}_d$ and let $x^*$ be a local maximum satisfying $\nabla f(x^*) = {\lambda^*} x^*$ and $\lambda^* > \bar{\lambda}_2 = \| \nabla^2 f(x^*) (I-x^*(x^*)^2) \|$. 
Then, there exists some $\delta > 0$ such that under the initial condition $1 - x_0^Tx^* < \delta$, the sequence of iterates $\{x_k\}_{k=0,1,\cdots}$ generated by SCI-PI satisfies
\begin{align*}
1 - (x_k^Tx^*)^2
\leq 
\prod_{t=0}^{k-1} \bigg( \frac{\bar{\lambda}_2}{\lambda^*} + \gamma_t \bigg)^{2} \left( 1 - (x_0^Tx^*)^2 \right), \, \,
\end{align*}
where
\begin{align*}
\frac{\bar{\lambda}_2}{\lambda^*} + \gamma_t < 1  \, \, \textup{for all } t \geq 0 \, \, \textup{and} \, \, \lim_{k \rightarrow \infty} \gamma_k = 0.
\end{align*}
Moreover, if $\nabla_i f = \partial f/\partial x_i$ has a continuous Hessian $H_i$ on an open set containing $\mathcal{B}_{d,\infty}$, we can explicitly write $\delta$ as
\begin{equation*}
\delta(\lambda^*, \bar{\lambda}_1, \bar{\lambda}_2, M) = \min \left \{ 
\left( \frac{\lambda^*}{\bar{\lambda}_1 + M} \right)^2, \, \left( 
\frac{\lambda^*-\bar{\lambda}_2}
{\bar{\lambda}_1 + \bar{\lambda}_2 + 2M} 
\right)^2, 1 \right \}
\end{equation*}
where $\bar{\lambda}_1 = \sqrt{2} \cdot \| \nabla^2 f(x^*) x^*\|$ and 
\begin{equation*}
M = \hspace{-2mm} \max_{x \in \partial \mathcal{B}_d, \\ y^1, \cdots, y^d \in \mathcal{B}_{d,\infty}} \hspace{-2mm} \sqrt{\textstyle \sum_{i=1}^d ( x^T G_i(y^1,\cdots,y^d) x )^2}, \quad G_i(y^1,\cdots,y^d) = {\textstyle \sum_{j=1}^d v_{i,j} H_j(y^j)}.
\end{equation*}
\end{theorem}

\begin{proof}
By Proposition~\ref{prop:scale-inv}, the gradient of $f$ has the form of
\begin{align*}
\nabla f(x) = \sum_{i=1}^{m} \nabla g_i(x) + \sum_{j=1}^{n} \nabla h_j(x) = 
F(x) x,
\end{align*}
where
\begin{equation*}
F(x) = \sum_{i=1}^{m} \left( \frac{1}{p_i-1} \right) \nabla^2 g_i(x) 
- \sum_{j=1}^{n} \nabla^2 h_j(x).
\end{equation*}
By the KKT conditions, a local optimal solution $x^*$ is an eigenvector of $F(x^*)$.
Let $\{v_1,\ldots,v_d\}$ be a set of eigenvectors of $F(x^*)$ with $x^* = v_1$. Since $F(x^*)$ is real and symmetric, without loss of generality, we assume that $\{ v_1, \ldots, v_d\}$ form an orthogonal basis in $\mathbb{R}^d$. 

Since $f$ is twice continuously differentiable on an open set containing $\partial \mathcal{B}_d$, for $x \in \partial \mathcal{B}_d$, using the Taylor expansion of $\nabla f(x)^Tv_i$ at $x^*$, we have
\begin{equation}
\begin{aligned}
\label{proof:grad-f-SSI}
\nabla f(x)^Tv_i &= \nabla f(x^*)^Tv_i + (x - x^*)^T \nabla^2 f(x^*)v_i  + R_i(x)
\end{aligned}
\end{equation}
where $R_i(x) = o(\|x-x^*\|)$.
Using \eqref{proof:grad-f-SSI} with $i=1$ and $\nabla f(x^*) = \lambda^* x^*$, we obtain
\begin{equation}
\label{proof:extended-1-grad-f-prod-v-1}
\begin{aligned}
\nabla f(x)^T v_1 &= \lambda^* (x^*)^Tv_1 + (x-x^*)^T \nabla^2 f(x^*) v_1  + R_1(x) \\
&= \lambda^* + \alpha(x)
\end{aligned}
\end{equation}
where
\begin{equation*}
\alpha(x) = (x-x^*)^T \nabla^2 f(x^*) v_1  + R_1(x) = o(\sqrt{\|x-x^*\|}).
\end{equation*}

Using \eqref{proof:grad-f-SSI} and $\nabla f(x^*) = \lambda^* x^*$ for $2 \leq i \leq d$, we have
\begin{align*}
\nabla f(x)^T v_i &= \lambda^* (x^*)^Tv_i + (x - x^*)^T \nabla^2 f(x^*) v_i  + R_i(x) \\
&= (x -x^*)^T \nabla^2 f(x^*) v_i  + R_i(x),
\end{align*}
resulting in
\begin{align}
\label{proof:grad-x-product-vi-square-sum-SSI}
\sum_{i=2}^d (\nabla f(x)^T v_i)^2 &= \sum_{i=2}^d \left( (x-x^*)^T \nabla^2 f(x^*) v_i + R_i(x) \right)^2.
\end{align}
Let $\bar{R}_2 (x) = {\textstyle \max_{2 \leq i \leq d} |R_i(x)|}$. Note that $\bar{R}_2 (x) = o(\|x-x^*\|)$.

From $x^* = v_1$ and the fact that $\{v_1,\ldots,v_d\}$ forms an orthogonal basis in $\mathbb{R}^d$, we have
\begin{align*}
\sum_{i=2}^d \left( (x-x^*)^T \nabla^2 f(x^*) v_i \right)^2 &= \| \nabla^2 f(x^*) (x-x^*) \|_2^2 - \left( (x-x^*)^T \nabla^2 f(x^*) v_1 \right)^2 \\
&= (x-x^*)^T \nabla^2 f(x^*) \left( I - x^* (x^*)^T \right) \nabla^2 f(x^*) (x-x^*) \\
&= (x-x^*)^T \nabla^2 f(x^*) \left( I - x^* (x^*)^T \right)^2 \nabla^2 f(x^*) (x-x^*).
\end{align*}
Since
\begin{align*}
\| \nabla^2 f(x^*) \left( I - x^* (x^*)^T \right)^2 \nabla^2 f(x^*) \|
&=
\| \left( I - x^* (x^*)^T \right) \nabla^2 f(x^*) \|^2
\\
&=
\| \nabla^2 f(x^*) \left( I - x^* (x^*)^T \right) \|^2,
\end{align*}
we have
\begin{align}
\label{proof:grad-x-product-vi-square-sum-bound-SSI}
\sum_{i=2}^d \left( (x-x^*)^T \nabla^2 f(x^*) v_i \right)^2 
\leq 
\bar{\lambda}_2^2 \|x - x^* \|^2.
\end{align}
Also, from \eqref{proof:grad-x-product-vi-square-sum-bound-SSI} and the Cauchy-Schwartz inequality, we obtain
\begin{align}
\label{proof:grad-x-product-vi-square-sum-bound-SSI-cross}
\sum_{i=2}^d (x-x^*)^T \nabla^2 f(x^*) v_i \leq \sum_{i=2}^d |(x-x^*)^T \nabla^2 f(x^*) v_i| \leq \bar{\lambda}_2 \sqrt{d} \|x - x^* \|.
\end{align}
Using \eqref{proof:grad-x-product-vi-square-sum-bound-SSI} and \eqref{proof:grad-x-product-vi-square-sum-bound-SSI-cross} for \eqref{proof:grad-x-product-vi-square-sum-SSI}, we obtain
\begin{align*}
\sum_{i=2}^d (\nabla f(x)^T v_i)^2 
\leq \bar{\lambda}_2^2 \|x - x^* \|^2 + 2 \bar{\lambda}_2 \bar{R}_2(x) \sqrt{d} \|x - x^* \| + d (\bar{R}_2(x))^2,
\end{align*}
resulting in
\begin{equation}
\label{proof:grad-x-product-vi-square-sum-final-SSI}
\begin{aligned}
\sum_{i=2}^d (\nabla f(x)^T v_i)^2 \leq \left( \bar{\lambda}_2 \|x - x^* \|^2 + \beta(x) \right)^2
\end{aligned}
\end{equation}
where
\begin{align*}
\beta(x) = \sqrt{d} \bar{R}_2(x) = o(\|x-x^*\|).
\end{align*}
By \eqref{proof:extended-1-grad-f-prod-v-1}, \eqref{proof:grad-x-product-vi-square-sum-final-SSI}, and Lemma~\ref{lemma:SCI-PI-local-convergence-small-delta}, we obtain the first part of the desired result.

Next, we assume that $\nabla_i f$ has a continuous Hessian $H_i$. By the Taylor theorem, we have
\begin{align}
\nabla_i f(x_k) = \nabla_i f(x^*) + \nabla \nabla_{i} f(x^*) (x_k - x^*) + \frac{1}{2} \left( x_k -x^* \right)^T H_i(\hat{x}_k^i) \left( x_k - x^* \right)
\label{eq:Taylor-gi-SSI}
\end{align}
for some $\hat{x}_{k}^{i} \in \mathcal{N} (x_k, x^*)$.

Taking the steps used to derive \eqref{proof:SCI-PI-second-order-bound} and \eqref{proof:SCI-PI-second-order-sum-bound} in the proof of Theorem~\ref{thm:GPM-local-convergence}, we can derive the same inequalities
\begin{equation}
\frac{1}{2}
\left | (x_k-x^*)^T G_i(\hat{x}_{k}^{1},\cdots,\hat{x}_{k}^{d}) (x_k - x^*) \right |
\leq \frac{1}{2} M \|x_k-x^*\|^2 \label{proof:SCI-SSI-second-order-bound}
\end{equation}
and
\begin{equation}
\frac{1}{4} \sum_{i=2}^d \big[ (x_k-x^*)^T G_i(\hat{x}_{k}^{1},\cdots,\hat{x}_{k}^{d}) (x_k - x^*) \big]^2
\leq \frac{M^2}{4} \|x_k-x^*\|^4. \label{proof:SCI-SSI-second-order-sum-bound}
\end{equation}
Using \eqref{eq:Taylor-gi-SSI}, \eqref{proof:SCI-SSI-second-order-sum-bound} and that $x^* = v_1$, we have
\begin{align*}
\begin{aligned}
\nabla f(x_k)^T v_1 &\geq \nabla f(x^*)^T x^* + (x_k - x^*)^T \nabla^2 f(x^*) x^* -  \frac{M}{2}  \| x_k - x^* \|^2
\end{aligned}
\end{align*}
resulting in 
\begin{equation}
\label{proof:product-v1-final-SSI}
\begin{aligned}
\nabla f(x_k)^T v_1 &\geq  \lambda^* - \| \nabla^2 f(x^*) x^* \| \sqrt{2(1-x_k^Tx^*)}- M (1 - x_k^Tx^*) \\
&=  \lambda^* - \bar{\lambda}_1 \sqrt{(1-x_k^Tx^*)}- M (1 - x_k^Tx^*)
\end{aligned}
\end{equation}
For $2 \leq i \leq d$, we have 
\begin{align} 
\nabla f(x_k)^T v_i &\leq \nabla f(x^*)^T v_i + (x_k - x^*)^T \nabla^2 f(x^*) v_i + \frac{1}{2} (x_k-x^*)^T G_i(\hat{x}_{k}^{1},\cdots,\hat{x}_{k}^{d}) (x_k - x^*) \nonumber \\
&= \lambda^* (x^*)^T v_i + (x_k - x^*)^T \nabla^2 f(x^*) v_i + \frac{1}{2} (x_k-x^*)^T G_i(\hat{x}_{k}^{1},\cdots,\hat{x}_{k}^{d}) (x_k - x^*) \nonumber \\
&= (x_k - x^*)^T \nabla^2 f(x^*) v_i + \frac{1}{2} (x_k-x^*)^T G_i(\hat{x}_{k}^{1},\cdots,\hat{x}_{k}^{d}) (x_k - x^*).
\label{proof:product-vi-SSI-final}
\end{align}
From \eqref{proof:product-vi-SSI-final}, \eqref{proof:grad-x-product-vi-square-sum-bound-SSI}, \eqref{proof:SCI-SSI-second-order-bound}, \eqref{proof:SCI-SSI-second-order-sum-bound} and the Cauchy-Shwartz inequality, we have
\begin{align}
\sum_{i=2}^d (\nabla f(x_k)^T v_i)^2 &\leq \sum_{i=2}^d \Big( (x_k - x^*)^T \nabla^2 f(x^*) v_i + \frac{1}{2} (x_k-x^*)^T G_i(\hat{x}_{k}^{1},\cdots,\hat{x}_{k}^{d}) (x_k - x^*) \Big)^2 \nonumber \\
&\leq \Big( \bar{\lambda}_2 \| x_k - x^* \| + \frac{M}{2}  \| x_k - x^* \|^2 \Big)^2.
\label{proof:rewrite2-SSI}
\end{align}
Using \eqref{proof:product-v1-final-SSI}, \eqref{proof:rewrite2-SSI}, and Lemma~\ref{lemma:SCI-PI-local-convergence-explicit-delta} with
$$
A = \lambda^*, \, B = M, \, C = \bar{\lambda}_1 , \, D = 0, \, E = \bar{\lambda}_2, \, F = M,
$$
we obtain the desired result.
\end{proof}

Note that $\bar{\lambda}_1$ has the additional $\sqrt{2}$ factor which comes from the fact that $x^*$ is not necessarily an eigenvector of $\nabla^2 f(x^*)$. Nonetheless, the asymptotic convergence rate in Theorem~\ref{thm:SCI-PI-local-convergence-SSI} provides a generalization of the convergence rate in Theorem~\ref{thm:GPM-local-convergence}.

\subsection{Block Scale Invariant Problems}
Next, consider a class of optimization problems having the form of
\begin{equation*}
\textrm{\rm maximize} \quad f(x,y) \quad \textrm{\rm subject to} \quad x \in \partial\mathcal{B}_{d_1},\, y \in \partial\mathcal{B}_{d_2}
\end{equation*}
where $f : \mathbb{R}^{d_1+d_2} \to \mathbb{R}$ is scale invariant in $x$ for fixed $y$ and vice versa. Some examples of block scale invariant problems are given next.

\begin{example}[Semidefinite Programming (SDP) \citep{erdogdu2018convergence}]
Let $A, X \in \mathbb{R}^{n \times n}$. Given an SDP problem
\begin{equation*}
\textrm{\rm maximize} \quad \langle A, X \rangle \quad \textrm{\rm subject to} \quad X_{ii} = 1, \, \, i \in \{1,2,\cdots,n\}, \, \, X \succeq 0,
\end{equation*}
the Burer-Monteiro approach \citep{burer2003nonlinear} yields the following block scale invariant problem
\begin{equation*}
\begin{aligned}
\textrm{\rm maximize} \quad \langle A, \sigma \sigma^T \rangle \quad \textrm{\rm subject to} \quad \| \sigma_i \| = 1, \, \, i \in \{1,2,\cdots,n\}.
\end{aligned}
\end{equation*}
\end{example}
\begin{example}[Kullback-Leibler (KL) divergence NMF] \label{eg:klnmf}
The KL-NMF problem \citep{fevotte2011algorithms, lee2001algorithms, wang2013nonnegative} is defined as
\baa \label{eqn:kldiv_nmf}
\textrm{\rm minimize} {}& \quad D_{KL}(V \| WH) \triangleq {\textstyle \sum_{i,j}} \left[V_{ij} \log \frac{V_{ij}}{\sum_k W_{ik} H_{kj}} - V_{ij} + {\textstyle \sum_k} W_{ik} H_{kj} \right] \\
\textrm{\rm subject to} {}& \quad W_{ik}\geq 0, \ H_{kj} \geq 0,\ i = 1,\cdots,n,\ j = 1,\cdots,m,\ k =1,\cdots,K.
\eaa
\end{example}
Many popular algorithms for the KL-NMF problem are based on alternate minimization of $W$ and $H$.
Given $W \geq 0$ and $j \in\{1,\cdots,m\}$, we consider a subproblem such that
\baa \label{prob:kldivsubproblem}
\textrm{\rm minimize} \quad f_{KL}(h) = {\textstyle \sum_i} \left[v_{i} \log \frac{v_{i}}{\sum_k W_{ik} h_k} - v_{i} + {\textstyle \sum_k} W_{ik} h_k \right] \, \, \textrm{\rm subject to}  \, \, h_k \geq 0
\eaa
where we let $v_i = V_{ij}$ and $h_k = H_{kj}$ as the objective is decomposed into $m$ separate subproblems.
Note that the KL-NMF problem in the form of \eqref{eqn:kldiv_nmf} is not a block scale invariant problem.
However, using a novel reformulation, we show that the KL divergence NMF subproblem is indeed a scale invariant problem.
\begin{lemma} \label{lem:kldiv}
The KL-NMF subproblem \eqref{prob:kldivsubproblem} is equivalent to the following scale invariant problem
\baa \label{prob:mixsqpsubproblem}
\textrm{\rm maximize} \quad -{\textstyle \sum_i} v_i \log {\textstyle \sum_k} W_{ik} \bar{h}_k \quad \textrm{\rm subject to} \quad {\textstyle \sum_k} \bar{h}_k = 1,\ \ \bar{h}_k \geq 0,
\eaa
with the relationship $(\sum_i v_i) \bar{h}_k = (\sum_i W_{ik}) h_k$.
\end{lemma}

\begin{proof}
Since a log-linear function is concave, \eqref{prob:kldivsubproblem} is a convex problem in $h$. Consider the Lagrangian of the original problem
\baa
\mathcal{L}(h,\lambda) = f_{KL}(h) - {\textstyle \sum_k} \lambda_k h_k
\eaa
where $\lambda \geq 0$. By the first-order KKT conditions, we must have
\baa
\label{eq:NMF-KKT}
\nabla_k f_{KL}(h^*) = \lambda_k^*,\quad  \lambda_k^* h_k^* = 0,\ \ \forall k = 1,\cdots,K
\eaa
at an optimal solution $(h^*,\lambda^*)$. Since \eqref{eq:NMF-KKT} implies $\sum_{k} h_k^* \lambda_k^* = 0$, we have
\begin{align*}
\sum_{k} h_k^* \lambda_k^* = \sum_{k} h_k^* \nabla_k f_{KL}(h^*) =-{\textstyle} \sum_{i,k} \frac{v_i W_{ik} h_k^* }{\sum_{k'} W_{ik'}h_{k'}^*} + {\textstyle} \sum_{i,k} W_{ik} h_k^*,
\end{align*}
resulting in
\baa
\textstyle \sum_i v_i = \sum_{i,k} W_{ik} h_k^*.
\eaa
Next, we show that
\baa
\label{prob:NMF-sub-reformulation}
{\rm minimize}  \, \, \, f_{SCI}(h) = {\textstyle \sum_i} v_i \log \frac{v_i}{{\textstyle \sum_k} W_{ik} h_k } \, {\rm subject\ to} {}& \, \sum_i v_i = \sum_{i,k} W_{ik} h_k, \, h_k \geq 0.
\eaa
is equivalent to the original subproblem \eqref{prob:kldivsubproblem}, due to the following:
\begin{enumerate}[wide, labelwidth=!, labelindent=0pt]
    \item It always satisfies $f_{SCI}^* \geq f_{KL}^*$ since \eqref{prob:NMF-sub-reformulation} has an additional constraint $\sum_i v_i = \sum_{i,k} W_{ik} h_k$ compared to \eqref{prob:kldivsubproblem}.
    \item A solution $h^*$ of \eqref{prob:kldivsubproblem} is a feasible point of \eqref{prob:NMF-sub-reformulation} since we have shown that $\sum_i v_i = \sum_{i,k} W_{ik} h_k^*$. This implies $f_{KL}^* \geq f_{SCI}^*$.
\end{enumerate}
Now, we can reparametrize $h$ by $\bar{h}$ so that $\sum_i v_i = \sum_{i,k} W_{ik} h_k$ if and only if $\sum_k \bar{h}_k = 1$, which yields the relationship between two variables $\bar{h}_k = h_k (\sum_{i} W_{ik}) /(\sum_i v_i)$. Note that \eqref{prob:mixsqpsubproblem} has the optimization problem as Example~\ref{eg:mixture} and thus a scale invariant problem.
\end{proof}
To solve block scale invariant problems, we consider an alternating maximization algorithm called \textit{block SCI-PI}, which repeats
\begin{equation} \label{alg:joint-GPM}
\begin{aligned}
x_{k+1} \leftarrow \nabla_x f(x,y_k) / \| \nabla_x f(x,y_k) \|, \quad y_{k+1} \leftarrow \nabla_y f(x_k,y) / \| \nabla_y f(x_k,y) \|.
\end{aligned}
\end{equation}
We present a local convergence result of block SCI-PI below.
\begin{theorem} \label{thm:joint-GPM-local-convergence}
Suppose that $f$ is twice continuously differentiable on an open set containing $\partial \mathcal{B}_{d_1} \times \partial \mathcal{B}_{d_2}$ and let $(x^*,y^*)$ be a local maximum satisfying
\begin{equation*}
\nabla_x f(x^*,y^*) = {\lambda^*} x^*, \, \lambda^* > \bar{\lambda}_2 = \max_{2 \leq i \leq d_1} |\lambda_i|, \, \nabla_y f(x^*,y^*) = {s^*} y^*, \, s^* > \bar{s}_2 = \max_{2 \leq i \leq d_2} |s_i|
\end{equation*}
where $(\lambda_i,v_i)$ and $(s_i,u_i)$ are eigen-pairs of $\nabla_x^2 f(x^*,y^*)$ and $\nabla_y^2 f(x^*,y^*)$, respectively with $x^* = v_1$ and $y^* = u_1$.
If 
\begin{equation*}
\nu^2 = \| \nabla_{yx} f(x^*,y^*) \|^2 < (\lambda^* - \bar{\lambda}_2)(s^* - \bar{s}_2),
\end{equation*}
then for the sequence of iterates $\{(x_k,y_k)\}_{k=0,1,\cdots}$ generated by \eqref{alg:joint-GPM}, there exists some $\delta > 0$ such that if 
$
\max
\{
|1-x_0^Tx^*|, |1-y_0^Ty^*|
\}
< \delta,
$
then we have 
\begin{equation*}
\left \| 
\Delta_k
\right \| \leq
\textstyle{\prod_{t=0}^{k-1}} \left( \rho + \gamma_t \right)
\left \|
\Delta_0 \right\|    
 \, \, \textup{and} \, \, 
\textup{lim}_{k \rightarrow \infty} \gamma_k = 0
\end{equation*}
where
\begin{align*}
\Delta_k
=
\begin{bmatrix}
\sqrt{1-(x_k^Tx^*)^2} \\ 
\sqrt{1-(y_k^Ty^*)^2}
\end{bmatrix}
, \, \, \rho = \frac{1}{2} \left[ \frac{\bar{\lambda}_2}{\lambda^*} + \frac{\bar{s}_2}{s^*} + \sqrt{ \left[ \frac{\bar{\lambda}_2}{\lambda^*} - \frac{\bar{s}_2}{s^*} \right]^2 + \frac{4\nu^2}{\lambda^* s^*} } \right] < 1.
\end{align*}
\end{theorem}

\begin{proof}
From Lemma~\ref{lemma:GPM-x} with $w=x_k$, $z=y_k$, we have
\begin{equation*}
1 - \frac{(\nabla_x f(x_k,y_k)^Tx^*)^2}{\| \nabla_x f(x_k,y_k) \|^2} \leq \left( \frac{\bar{\lambda}_2}{\lambda^*} \sqrt{1-(x_k^Tx^*)^2} + \frac{\nu}{\lambda^*} \| y_k - y^* \| 
+
\theta^x(x_k,y_k) \right)^2.
\end{equation*}
Since 
\begin{equation*}
x_{k+1} = \frac{\nabla_x f(x_k,y_k)}{\| \nabla_x f(x_k,y_k) \|},
\end{equation*}
we obtain
\begin{equation*}
\sqrt{1 - (x_{k+1}^Tx^*)^2} \leq 
\frac{\bar{\lambda}_2}{\lambda^*} \sqrt{1-(x_k^Tx^*)^2} + \frac{\nu}{\lambda^*} \| y_k - y^* \| 
+
\theta^x(x_k,y_k).
\end{equation*}
Using
\begin{align*}
\|y_k-y^*\| = \sqrt{2(1-y_k^Ty^*)} = \left( 1 + \frac{1-y_k^Ty^*}{1+y_k^Ty^*+\sqrt{2(1+y_k^Ty^*)})}\right)  \sqrt{1-(y_k^Ty^*)^2},
\end{align*}
we have
\begin{align}
\label{proof:joint-GPM-recurrence-x-final}
\sqrt{1 - (x_{k+1}^Tx^*)^2} & \leq \frac{\bar{\lambda}_2}{\lambda^*} \sqrt{1-(x_k^Tx^*)^2} + \frac{\nu}{\lambda^*} \sqrt{1-(y_k^Ty^*)^2}
+
\bar{\theta}^x(x_k,y_k)
\end{align}
where
\begin{align*}
\bar{\theta}^x(x_k,y_k) &= \theta^x(x_k,y_k) + \left[ \frac{1-y_k^Ty^*}{1+y_k^Ty^*+\sqrt{2(1+y_k^Ty^*)})} \right]  \sqrt{1-(y_k^Ty^*)^2} 
= o \left( \left \|
\begin{bmatrix}
x_k - x^* \\
y_k - y^*
\end{bmatrix}
\right \| \right).
\end{align*}
Using Lemma~\ref{lemma:GPM-x} for $w=y_k$, $z=x_k$ and the definition of $y_{k+1}$, we have
\begin{align}
\label{proof:joint-GPM-recurrence-y-final}
\sqrt{1 - (y_{k+1}^Ty^*)^2} & \leq \frac{\nu}{s^*} \sqrt{1-(x_k^Tx^*)^2} + \frac{\bar{s}_2}{s^*} \sqrt{1-(y_k^Ty^*)^2} + \bar{\theta}^y(x_k,y_k)
\end{align}
where
\begin{align*}
\bar{\theta}^y(x_k,y_k) &= \theta^y(x_k,y_k) + \left[ \frac{1-x_k^Tx^*}{1+x_k^Tx^*+\sqrt{2(1+x_k^Tx^*)})} \right]  \sqrt{1-(x_k^Tx^*)^2}
= o \left( \left \|
\begin{bmatrix}
x_k - x^* \\
y_k - y^*
\end{bmatrix}
\right \| \right).    
\end{align*}
Combining \eqref{proof:joint-GPM-recurrence-x-final} and \eqref{proof:joint-GPM-recurrence-y-final}, we obtain
\begin{align}
\label{proof:joint-GPM-recurrence}
\begin{bmatrix} \sqrt{1 - (x_{k+1}^Tx^*)^2} \\[2pt] \sqrt{1 - (y_{k+1}^Ty^*)^2} \end{bmatrix} 
&\leq
\left[ \begin{matrix} \dfrac{\bar{\lambda}_2}{\lambda^*} & \dfrac{\nu}{\lambda^*} \\[10pt] \dfrac{\nu}{s^*} & \dfrac{\bar{s}_2}{s^*} \end{matrix} \right]
\begin{bmatrix} \sqrt{1 - (x_{k}^Tx^*)^2} \\[4pt] \sqrt{1 - (y_{k}^Ty^*)^2} \end{bmatrix}
+ 
\begin{bmatrix} \bar{\theta}^x(x_k,y_k) \\[2pt] \bar{\theta}^y(x_k,y_k) \end{bmatrix}
\\
&
\leq
(M+N(x_k,y_k))
\begin{bmatrix} \sqrt{1 - (x_{k}^Tx^*)^2} \\[2pt] \sqrt{1 - (y_{k}^Ty^*)^2} \end{bmatrix}
\label{proof:joint-GPM-recurrence-final}
\end{align}
where
\begin{align*}
M = \left[ \begin{matrix} \dfrac{\bar{\lambda}_2}{\lambda^*} & \dfrac{\nu}{\lambda^*} \\[10pt] \dfrac{\nu}{s^*} & \dfrac{\bar{s}_2}{s^*} \end{matrix} \right]
, \quad
\epsilon (x,y)
=
\dfrac{\max \{ \bar{\theta}^x(x,y),\bar{\theta}^y(x,y)\}}{\sqrt{2-x^Tx^*-y^Ty^*}},
\end{align*}
and
\begin{equation*}
N(x,y)
=
\frac{\epsilon(x,y)}{\sqrt{2-x^Tx^*-y^Ty^*}}
\begin{bmatrix} 
\sqrt{\dfrac{{1-x^Tx^*}}{{1+x^Tx^*}}}
&
\sqrt{\dfrac{{1-y^Ty^*}}{{1+y^Ty^*}}} \\[10pt] 
\sqrt{\dfrac{{1-x^Tx^*}}{{1+x^Tx^*}}}
&
\sqrt{\dfrac{{1-y^Ty^*}}{{1+y^Ty^*}}}
\end{bmatrix}
.
\end{equation*}

Note that the spectral radius $\rho$ of $M$ satisfies
\begin{align*}
\rho = \frac{1}{2} \left( \frac{\bar{\lambda}_2}{\lambda^*} + \frac{\bar{s}_2}{s^*} + \sqrt{ \left( \frac{\bar{\lambda}_2}{\lambda^*} - \frac{\bar{s}_2}{s^*} \right)^2 + \frac{4\nu^2}{\lambda^* s^*} } \right) < 1
\end{align*}
due to $\nu^2 < (\lambda^* - \bar{\lambda}_2)(s^* - \bar{s}_2)$. Also, for $i,j = 1,2$, we have
\begin{equation*}
\lim_{\substack{(x,y) \rightarrow (x^*,y^*)}} N_{ij}(x,y) = 0.
\end{equation*}
By Lemma~\ref{lemma:asymptotic-spectral-property}, there exists a sequence $\omega_t$ such that 
\begin{align*}
\| M^k \| = \prod_{t=0}^{k-1} (\rho + \omega_t) \quad \text{and} \quad \text{lim}_{t \rightarrow \infty} \omega_t = 0.
\end{align*}

Let
\begin{align*}
\tau = \min \{ k: \| M^k \| < 1\}, \quad 
\bar{\rho} = \frac{\|M^\tau\|+1}{2}, \quad
\rho_{\max} = \max_{1 \leq k \leq \tau} \| M^k \|.
\end{align*}
By Lemma~\ref{lemma:GPM-x}, we have
\begin{equation*}
\begin{aligned}
\nabla_x f(x,y)^Tv_1 = \lambda^* + (y-y^*)^T\nabla_{yx}^2 f(x^*,y^*)x^* + \alpha^x(x,y) \\
\nabla_y f(x,y)^Tu_1 = s^* + (x-x^*)^T\nabla_{xy}^2 f(x^*,y^*)y^* + \alpha^y(x,y)
\end{aligned}
\end{equation*}
where
\begin{equation*}
 \alpha^x(x,y) =  o \left( \left \|
\begin{bmatrix}
x - x^* \\
y - y^*
\end{bmatrix}
\right \| \right), \quad  \alpha^y(x,y) =  o \left( \left \|
\begin{bmatrix}
x - x^* \\
y - y^*
\end{bmatrix}
\right \| \right).
\end{equation*}
Therefore, there exists some $\delta_1 > 0$ such that if 
\begin{equation*}
x^Tx^* > 0, \quad y^Ty^*>0, \quad  \left \| \begin{bmatrix} \sqrt{1 - (x^Tx^*)^2} \\[2pt] \sqrt{1 - (y^Ty^*)^2} \end{bmatrix} \right \| < \delta_1,
\end{equation*}
then
\begin{equation}
\nabla_x f(x,y)^T v_1 > 0, \quad \nabla_y f(x,y)^T u_1 > 0.
\label{proof:joint-delta1-def}
\end{equation}

Also, since $N_{ij}(x,y) \rightarrow 0$ as $(x,y) \rightarrow (x^*,y^*)$ for $i,j=1,2$, there exists some $\delta_2 > 0$ such that if
\begin{equation*}
x^Tx^* > 0, \quad y^Ty^*>0, \quad \left \| \begin{bmatrix} \sqrt{1 - (x^Tx^*)^2} \\[2pt] \sqrt{1 - (y^Ty^*)^2} \end{bmatrix} \right \| < \delta_2,
\end{equation*}
then we have
\begin{equation}
\label{proof:joint-delta2-def}
\left \| \prod_{l=0}^{\tau-1} \big( M+N(\phi(x,y,l)) \big) \right \| < \bar{\rho}, \quad \max_{0 < m \leq \tau} \left \| \prod_{l=0}^{m-1} \big( M+N(\phi(x,y,l)) \big) \right \| < 1 +  \rho_{\max}
\end{equation}
where $\phi(x,y,l)$ denotes the vector after $l$ iterations of the algorithm starting with $(x,y)$.
To see this, let us define 
\begin{equation*}
g(x,y,m) = \left \| \prod_{l=0}^{m-1} \big( M+N(\phi(x,y,l)) \big) \right \|.
\end{equation*}
By \eqref{proof:joint-GPM-recurrence-final} and \eqref{proof:joint-delta1-def}, if $x \rightarrow x^*$ and $y \rightarrow y^*$, then for any $0 \leq l \leq \tau$, we have
\begin{equation*}
\phi(x,y,l) \rightarrow (x^*,y^*),
\end{equation*}
resulting in 
\begin{equation*}
g(x,y,m) \rightarrow \| M^{m} \|.
\end{equation*}
Therefore, there exists some $\delta_{2,\tau} > 0$ such that $g(x,y,\tau) < \bar{\rho}$. Also, for each $1 \leq m < \tau$, there exists some $\delta_{2,m} > 0$ such that $g(x,y,m) < 1 + \rho_{\max}$. Taking the minimum of $\delta_{2,m}$ for $1 \leq m \leq \tau$, we obtain $\delta_2$ satisfying \eqref{proof:joint-delta2-def}.

Let
\begin{equation*}
\delta = \frac{\bar{\delta}}{\sqrt{2}}, \quad \bar{\delta} = \min \left \{ \delta_1,  \frac{\delta_1}{1 + \rho_{\max}}, \delta_2, 1 \right \}, \quad N_k = N(x_k,y_k).
\end{equation*}
By mathematical induction, we show that for any $n \geq 0$, if
\begin{equation}
\label{proof:joint-induction-base-2}
x_{n\tau}^Tx^* > 0, \quad y_{n\tau}^Ty^*>0, \quad \Delta_{n\tau} < \bar{\delta},    
\end{equation}
then for $0\leq m \leq \tau$, we have
\begin{align}
\label{proof:joint-induction-main-2}
x_{n\tau+m}^Tx^* > 0, \quad y_{n\tau+m}^Ty^*>0, \quad \Delta_{n\tau+m} \leq (1+\rho_{\max}) \Delta_{n\tau}
< \delta_1.
\end{align}
By \eqref{proof:joint-induction-base-2}, it is obvious that we have \eqref{proof:joint-induction-main-2} for $m=0$. This proves the base case. Next, suppose that we have \eqref{proof:joint-induction-main-2} for $0 \leq m<\tau$. Then, by the definition of $\delta_1$, we have
\begin{equation*}
x_{n\tau + m+1}^Tx^* = x_{n\tau + m+1}^Tv_1 = \frac{\nabla_x f(x_{n\tau + m}, y_{n\tau + m})^Tv_1}{\| \nabla_x f(x_{n\tau + m}, y_{n\tau + m}) \|} > 0
\end{equation*}
and
\begin{equation*}
y_{n\tau + m+1}^Ty^* = y_{n\tau + m+1}^Tu_1 = \frac{\nabla_y f(x_{n\tau + m}, y_{n\tau + m})^Tu_1}{\| \nabla_y f(x_{n\tau + m}, y_{n\tau + m}) \|} > 0.
\end{equation*}
Also, by \eqref{proof:joint-GPM-recurrence-final}, \eqref{proof:joint-induction-base-2} and \eqref{proof:joint-delta2-def}, we have
\begin{align*}
\Delta_{n\tau+m+1}
\leq \left \| \prod_{l=0}^m \left( M + N_{n\tau+l} \right) \right \| \Delta_{n\tau} \leq (1+ \rho_{\max}) \Delta_{n\tau} < \delta_1.
\end{align*}
This completes the induction proof.

Suppose that $(x_0,y_0)$ satisfies $\max \{|1-x_0^Tx^*|,|1-y_0^Ty^*|\} < \delta$. Then, we have
\begin{equation}
x_0^Tx^* > 0, \quad y_0^Ty^*>0, \quad \Delta_0 < \bar{\delta}.
\label{proof:joint-induction2-base}
\end{equation}
Now, we show
\begin{align}
\label{proof:joint-induction2-main}
x_{n\tau}^Tx^* > 0, \quad y_{n\tau}^Ty^*>0, \quad \Delta_{n\tau} \leq \bar{\rho}^n \Delta_{0}.
\end{align}
For $n=0$, we have \eqref{proof:joint-induction2-main} by \eqref{proof:joint-induction2-base}. This proves the base case. Next, suppose that we have \eqref{proof:joint-induction2-main} for $n$. Then, since \eqref{proof:joint-induction2-main}  implies that $\Delta_{n\tau} \leq \bar{\rho}^n \Delta_{0} < \bar{\delta}$, by \eqref{proof:joint-induction-main-2}, we have
\begin{equation*}
x_{(n+1)\tau}^Tx^* > 0, \quad y_{(n+1)\tau}^Ty^* > 0.
\end{equation*}
Moreover, using \eqref{proof:joint-GPM-recurrence-final} and \eqref{proof:joint-delta2-def}, we have
\begin{align*}
\Delta_{(n+1)\tau}
\leq \left \| \prod_{l=0}^{\tau-1} \left( M + N_{n\tau+l} \right) \right \| \Delta_{n\tau} \leq \bar{\rho} \Delta_{n\tau} < \bar{\rho}^{n+1} \Delta_0,
\end{align*}
which completes the induction proof. By repeatedly applying \eqref{proof:joint-induction2-main}, we have 
\begin{equation*}
(x_{n\tau},y_{n\tau}) \rightarrow (x^*,y^*) \, \, \text{as} \, \, n \rightarrow \infty.
\end{equation*}
Furthermore, due to \eqref{proof:joint-induction-main-2}, we have 
\begin{equation*}
(x_{n\tau+m},y_{n\tau+m}) \rightarrow (x^*,y^*) \, \, \text{for every} \, \, 0<m \leq \tau,
\end{equation*}
indicating that 
\begin{equation*}
(x_k,y_k) \rightarrow (x^*,y^*).
\end{equation*}
This in turn implies that $N_k \rightarrow 0$. Letting
\begin{equation*}
\eta_k = \frac{\| \prod_{t=0}^{k} (M+N_t) \|}{\| \prod_{t=0}^{k-1} (M+N_t) \|} - \frac{\| M^{k+1} \|}{\| M^{k} \|}, \quad \gamma_k = \omega_k + \eta_k,
\end{equation*}
we have
\begin{align}
\label{proof:joint-Delta-norm-product-bound}
\left \| \prod_{t=0}^{k-1} (M+N_t) \right \| 
= \prod_{t=0}^{k-1} (\rho + \omega_t + \eta_t) = \prod_{t=0}^{k-1} (\rho + \gamma_t).
\end{align}
Since $\eta_k \rightarrow 0$ as $N_k \rightarrow 0$, we have $\lim \gamma_k = 0$. This concludes the proof.
\end{proof}

If $x$ and $y$ are independent ($\nu = 0$), we have $\rho = \max \, \{\bar{\lambda}_2/\lambda^*, \bar{s}_2/s^* \}$. Otherwise, $\rho$ increases as $\nu$ increases. Note that the result of Theorem~\ref{thm:GPM-local-convergence} can be restored by dropping $x$ or $y$ in Theorem~\ref{thm:joint-GPM-local-convergence}. While we consider the two-block case, the algorithm and the convergence analysis can be easily generalized to more than two blocks.

\subsection{Partially Scale Invariant Problems}
\label{subsec:restricted}
Lastly, we consider a class of optimization problems of the form
\begin{equation*}
\textrm{\rm maximize} \quad f(x,y) \quad \textrm{\rm subject to} \quad x \in \partial\mathcal{B}_{d_1}
\end{equation*}
where $f(x,y) : \mathbb{R}^{d_1+d_2} \to \mathbb{R}$ is a scale invariant function in $x$ for each $y \in \mathbb{R}^{d_2}$. 
A partially scale invariant problem has the form of \eqref{prob:main} with respect to $x$ once $y$ is fixed. If $x$ is fixed, we obtain an unconstrained optimization problem with respect to $y$.
\begin{example}[Gaussian Mixture Model (GMM)]
The GMM problem is defined as
\begin{equation*}
\textrm{\rm maximize} \quad \sum_{i=1}^n \log \sum_{k=1}^d x_k^2 \, \mathcal{N}(x_i;\mu_k,\Sigma_k) \quad \textrm{\rm subject to} \quad x \in \partial \mathcal{B}_{d}.
\end{equation*}
Note that the objective function is scale invariant in $x$ for fixed $\mu_k$ and $\Sigma_k$, and $\mu_k$ is unconstrained. If we assume some structure on $\Sigma_k$, estimation of $\Sigma_k$ can also be unconstrained. For general $\Sigma_k$, semi-positive definiteness is necessary for $\Sigma_k$.
\end{example}

To solve partially scale invariant problems, we consider an alternative maximization algorithm based on SCI-PI and the gradient method as
\begin{equation} \label{alg:restricted-GPM}
x_{k+1} \leftarrow \nabla_x f(x_{k},y_k)/\| \nabla_x f(x_{k},y_k) \|, \quad y_{k+1} \leftarrow y_k + \alpha \nabla_{y} f(x_k, y_k).
\end{equation}
While the gradient method is used in \eqref{alg:restricted-GPM}, any method for unconstrained optimization can replace it. We present a convergence analysis of \eqref{alg:restricted-GPM} below.
\begin{theorem}
\label{thm:restricted-GPM-local-convergence}
Suppose that $f(x,y)$ is scale invariant in $x$ for each $y \in \mathbb{R}^{d_2}$, $\mu$-strongly concave in $y$ with an $L$-Lipschitz continuous $\nabla_y f(x,y)$ for each $x \in \partial \mathcal{B}_{d_1}$, and three-times continuously differentiable on an open set containing $\partial \mathcal{B}_{d_1} \times \mathbb{R}^{d_2}$. Let $(x^*,y^*)$ be a local maximum satisfying
\begin{equation*}
\nabla f(x^*) = {\lambda^*} x^*, \, \, \lambda^* > \bar{\lambda}_2 = {\textstyle \max_{2 \leq i \leq d}} |\lambda_i|
\end{equation*}
where $(\lambda_i,v_i)$ is an eigen-pair of $\nabla^2 f(x^*)$ with $x^* = v_1$. If 
\begin{equation*}
\nu^2 = \| \nabla_{yx}^2 f(x^*,y^*) \|^2 < \mu (\lambda^* - \bar{\lambda}_2),
\end{equation*}
then for the sequence of iterates $\{(x_k,y_k)\}_{k=0,1,\cdots}$ generated by \eqref{alg:restricted-GPM} with $\alpha = 2/(L+\mu)$, there exists some $\delta > 0$ such that if $\max \{|1-x_0^Tx^*|, \|y-y^* \|\} < \delta$,
then we have 
\begin{equation*}
\left \| 
\Delta_k
\right \| \leq
\textstyle{\prod_{t=0}^{k-1}} \left( \rho + \gamma_t \right)
\left \|
\Delta_0 \right\|
\, \,
\textup{and}
\, \,
\textup{lim}_{k \rightarrow \infty} \gamma_k = 0
\end{equation*}
where
\begin{align*}
\Delta_k
=
\begin{bmatrix}
\sqrt{1-(x_k^Tx^*)^2} \\ 
\| y_k - y^* \|
\end{bmatrix}
, \,
\rho = \frac{1}{2} \left[ \frac{\bar{\lambda}_2}{\lambda^*} + \frac{L-\mu}{L+\mu} + \sqrt{\left[ \frac{\bar{\lambda}_2}{\lambda^*} - \frac{L-\mu}{L+\mu} \right]^2 + \frac{8 \nu^2}{\lambda^* (L+\mu)}} \right] < 1.
\end{align*}
\end{theorem}

\begin{proof}
Using Lemma~\ref{lemma:GPM-x} for $w=x_k, \, z=y_k$ and the definition of $x_{k+1}$, we have
\begin{align}
\label{proof:restricted-GPM-recurrence-x}
\sqrt{1 - (x_{k+1}^Tx^*)^2} & \leq \frac{\bar{\lambda}_2}{\lambda^*} \sqrt{1-(x_k^Tx^*)^2} + \frac{\nu}{\lambda^*} \| y_k - y^* \| 
+
\theta^x(x_k,y_k).
\end{align}
where
\begin{equation*}
\theta^x(x_k,y_k) = o \left( \left \|
\begin{bmatrix}
x_k - x^* \\
y_k - y^*
\end{bmatrix}
\right \| \right).
\end{equation*}
By Lemma~\ref{lemma:GPM-y} with $w = x_k, z=y_k$, we also have
\begin{align}
\label{proof:restricted-GPM-recurrence-y}
\| y_{k+1} - y^*\| \leq \left( \frac{2\nu}{L+\mu} \right) \|x_k-x^*\| + \left( \frac{L-\mu}{L+\mu} \right) \| y_k - y^* \| + \theta^y(x_k,y_k).
\end{align}
Using
\begin{align*}
\bar{\theta}^y (x_k,y_k) = \theta^y(x_k,y_k) + \left[ \frac{1-x_k^Tx^*}{1+x_k^Tx^*+\sqrt{2(1+x_k^Tx^*)})} \right]  \sqrt{1-(x_k^Tx^*)^2} = o \left( \left \|
\begin{bmatrix}
x_k - x^* \\
y_k - y^*
\end{bmatrix}
\right \| \right),
\end{align*}
we can write \eqref{proof:restricted-GPM-recurrence-y} as
\begin{align}
\label{proof:restricted-GPM-recurrence-y-final}
\| y_{k+1} - y^*\| \leq \left( \frac{2\nu}{L+\mu} \right) \sqrt{1-(x_k^Tx^*)^2} + \left( \frac{L-\mu}{L+\mu} \right) \| y_k - y^* \| + \bar{\theta}^y (x_k,y_k).
\end{align}
Combining \eqref{proof:restricted-GPM-recurrence-x} and \eqref{proof:restricted-GPM-recurrence-y-final}, we obtain
\begin{align}
\label{proof:restricted-GPM-recurrence}
\begin{bmatrix} \sqrt{1 - (x_{k+1}^Tx^*)^2} \\[4pt] \| y_{k+1} - y^* \| \end{bmatrix} 
&\leq
\left[ \begin{matrix} \dfrac{\bar{\lambda}_2}{\lambda^*} & \dfrac{\nu}{\lambda^*} \\[10pt] \dfrac{2\nu}{L+\mu} & \dfrac{L-\mu}{L+\mu} \end{matrix} \right]
\begin{bmatrix} \sqrt{1 - (x_{k}^Tx^*)^2} \\[4pt] \| y_{k} - y^* \| \end{bmatrix}
+ 
\begin{bmatrix}
{\theta}^x (x_k,y_k) \\[5pt] \bar{\theta}^y (x_k,y_k)
\end{bmatrix} \\
&\leq
\left( M + N(x_k,y_k) \right)
\begin{bmatrix} \sqrt{1 - (x_{k}^Tx^*)^2} \\[2pt] \| y_k - y^* \| \end{bmatrix} \label{proof:restricted-GPM-recurrence-final}
\end{align}
where
\begin{align*}
M = \left[ \begin{matrix} \dfrac{\bar{\lambda}_2}{\lambda^*} & \dfrac{\nu}{\lambda^*} \\[10pt] \dfrac{2\nu}{L+\mu} & \dfrac{L-\mu}{L+\mu} \end{matrix} \right]
, \quad 
\epsilon (x,y)
=
\frac{\max \{ {\theta}^x(x,y), \bar{\theta}^y(x,y) \}}{\sqrt{1-x^Tx^*+\|y-y^*\|^2}}
\end{align*}
and
\begin{equation*}
N(x,y)
=
\frac{\epsilon(x,y)}{\sqrt{1-x^Tx^*+\|y-y^*\|^2}}
\begin{bmatrix} 
\sqrt{\dfrac{{1-x^Tx^*}}{{1+x^Tx^*}}}
&
\| y - y^* \|
\\[15pt] 
\sqrt{\dfrac{{1-x^Tx^*}}{{1+x^Tx^*}}}
&
\| y - y^* \|
\end{bmatrix}
.
\end{equation*}
Since $\nu^2 < \mu (\lambda^* - \bar{\lambda}_2)$, the spectral radius $\rho$ of $M$ satisfies
\begin{equation*}
\rho = \frac{1}{2} \left( \frac{\bar{\lambda}_2}{\lambda^*} + \frac{L-\mu}{L+\mu} + \sqrt{\left( \frac{\bar{\lambda}_2}{\lambda^*} - \frac{L-\mu}{L+\mu} \right)^2 + \frac{8 \nu^2}{\lambda^* (L+\mu)}} \right) < 1.
\end{equation*}
The rest of the proof is the same as the steps taken in the proof of Theorem~\ref{thm:joint-GPM-local-convergence}.
\end{proof}

As in the result of Theorem~\ref{thm:joint-GPM-local-convergence}, the rate $\rho$ increases as $\nu$ increases and is equal to $\max \, \{ \bar{\lambda}_2/\lambda^*, (L-\mu)/(L+\mu)\}$ when $\nu = 0$. Also, by dropping $y$, we can restore the convergence result of Theorem~\ref{thm:GPM-local-convergence}.

\section{Numerical Experiments}
\label{sec:numerical-experiments}
We test the proposed algorithms on real-world data sets. All experiments are implemented on a standard laptop (2.6 GHz Intel Core i7 processor and 16GM memory) using the Julia programming language. Let us emphasize that scale invariant problems frequently appear in many important applications in statistics and machine learning. We select 3 important applications, KL-NMF, GMM and ICA. A description of the data sets is provided below.

\subsection{Description of Data Sets}
\begin{table}[h!]
	\caption{A brief summary of data sets used for KL-NMF} \label{table:1}
	\footnotesize
	\centerline{
		\begin{tabular}{|r || r | r | r | r | r |}
		Name & \# of samples & \# of features & \# of nonzeros & Sparsity \\
		\midrule
		WIKI & 8,274 & 8,297 & 104,000 & 
		0.999 \\
		NIPS & 1,500 & 12,419 & 280,000 & 
		0.985 \\
				KOS & 3,430 & 6,906 & 950,000 &
		0.960  \\
				WT & 287 & 19,200 & 5,510,000 &
		0.000 
		\end{tabular}
		}
\end{table}
For KL divergence nonnegative matrix factorization (Section~\ref{subsec:nmf}), we use 4 public real data sets available online\footnote{These 4 data sets are retrieved from \url{https://www.microsoft.com/en-us/research/project}, \url{https://archive.ics.uci.edu/ml/datasets/bag+of+words}, and \url{https://snap.stanford.edu/data/wiki-Vote.html}} and summarized in Table~\ref{table:1}. Waving Trees (WT) has 287 images, each having $160 \times 120$ pixels. KOS and NIPS are sparse, large matrices implemented for topic modeling. WIKI is a large binary matrix having values $0$ or $1$ representing the adjacency matrix of a directed graph.
\begin{table}[h!]
	\caption{A brief summary of data sets used for GMM} \label{table:2}
	\footnotesize
	\centerline{
		\begin{tabular}{| r | r || r | r | r | r | r |}
		Name &  \# of classes & \# of samples & Dimension \\
		\midrule
		Sonar & 2 & 208 & 60 \\
	    Ionosphere & 2 &  351 & 34 \\
		HouseVotes84 & 2 & 435 & 16 \\
		BrCancer & 2& 699 & 10 \\
		PIDiabetes & 2& 768 & 8 \\
		Vehicle & 4& 846 & 18 \\
		Glass & 6 & 214 & 9 \\
		Zoo & 7 &  101 & 16\\
		Vowel & 11 & 990 & 10 \\
		Servo & 51 & 167 & 4  \\
		\end{tabular}
		}
\end{table}

For GMM (Section~\ref{subsec:gmm}), we use 10 public real data sets, corresponding to all small and moderate data sets provided by the {\tt mlbench} package in R. We select data sets for multi-class classification problems and run EM and SCI-PI for the given number of classes without class labels. In Table~\ref{table:2}, the sample size varies from $101$ to $990$, the dimension varies from $2$ to $60$, and the number of classes varies from $2$ to $51$. Only a small portion of entries are missing, if missing data exists, and we simply impute by mean.
\begin{table}[h!]
	\caption{A brief summary of data sets used for ICA} \label{table:3}
	\footnotesize
	\centerline{
		\begin{tabular}{| r || r | r | r | r | r |}
		Name & \# of samples & \# of features \\
		\midrule
		Wine & 178 & 14 \\
		Soybean & 683 &  35\\
		Vehicel & 846 & 18 \\
		Vowel & 990 &  10 \\
		Cardio & 2,126 &  22 \\
		Satellite & 6,435 &  37 \\
		Pendigits & 10,992 & 17 \\
		Letter & 20,000 &  16 \\
		Shuttle & 58,000  &  9
		\end{tabular}
		}
\end{table}

For ICA, discussed also in Section 5.3, we use 9 public data sets (see Table~\ref{table:3}) from the UCI Machine Learning repository\footnote{\url{https://archive.ics.uci.edu/ml/index.php}}. The sample size varies from $178$ to 58,000 and the dimension varies from $9$ to $37$.
\subsection{KL-divergence Nonnegative Matrix Factorization} \label{subsec:nmf}
We perform experiments on the KL-divergence NMF (KL-NMF) problem \eqref{eqn:kldiv_nmf} described in Example~\ref{eg:klnmf}. Let us recall that the original KL-NMF problem can be solved via block SCI-PI where in each iteration the algorithm solves the subproblem of the form \eqref{prob:mixsqpsubproblem}. Our focus is to compare this algorithm with other well-known alternating minimization algorithms listed below, updating $H$ and $W$ alternatively. To lighten the notation, let $\odot$, $\oslash$ and $(\cdot)^{\odot 2}$ denote element-wise product, division and square, respectively. We let $z = V \oslash (Wh)$ and $\mathds{1}_n$ denote a vector of ones. 
\begin{itemize}
\item Projected gradient descent (PGD): It iterates $h^{\rm new} \leftarrow h - \eta \odot W^T(z - \mathds{1}_n)$ followed by projection onto the simplex, where $\eta \propto h$ is an appropriate learning rate \citep{lin2007projected}.
\item Multiplicative update (MU): A famous multiplicative update algorithm is originally suggested by \citep{lee2001algorithms}, which iterates $h^{\rm new} \leftarrow h \odot (W^Tz)\oslash (W^T \mathds{1}_n)$ and is learning rate free.
\item Our method (SCI-PI): It iterates $h^{\rm new} \leftarrow h \odot ({\sigma} + W^T z)^{\odot 2}$ and rescales $h$, where $\sigma$ is a shift parameter. We simply use $\sigma = 1$ for preconditioning.
\item Sequential quadratic programming (MIXSQP): It exactly solves each subproblem via a convex solver {\tt mixsqp} \citep{kim2018fast}. This algorithm performs sequential non-negative least squares.
\end{itemize}

\paragraph{KL-NMF Subproblem} Note that the KL-NMF subproblem \eqref{prob:mixsqpsubproblem} has exactly the same form of the estimation of mixture proportions \eqref{eqn:mixture-reform} described in the Example~\ref{eg:mixture}.

\begin{figure}[h]
    \centering
    \centerline{\includegraphics[width=6in]{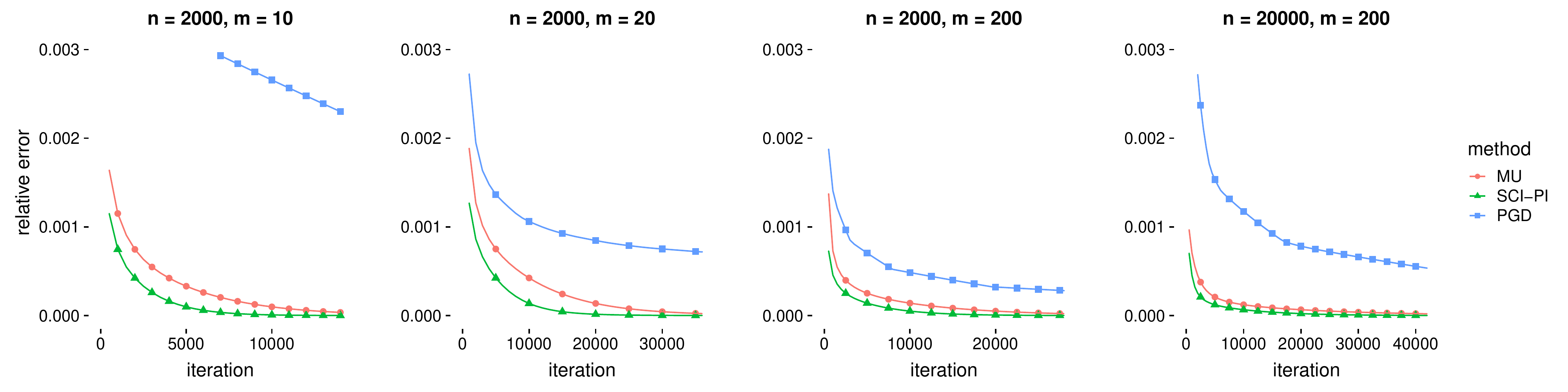}}
    \caption{Convergence of 3 algorithms for the KL-NMF subproblem. $n,m$ : the number of samples/features of the data matrix.}
    \label{fig:mixprop}
\end{figure}

To study the convergence rate for the KL-NMF subproblems, we use the 4 data sets studied in \citet{kim2018fast}. We study MU, PGD and SCI-PI since they have the same order of computational complexity per iteration, but omit MIXSQP since it is a second-order method which cannot be directly compared. For PGD, the learning rate is optimized by grid search. The stopping criterion is $\| f(x_k) - f^* \| \leq 10^{-6} f^*$ where $f^*$ is the solution obtained by MIXSQP after extensive computation time. The average runtime for aforementioned 3 methods are 33, 33 and 30 seconds for 10,000 iterations, respectively. The result is shown in Figure~\ref{fig:mixprop}\footnote{\label{note1}For each evaluation, we randomly draw 10 initial points and report the averaged relative errors with respect to $f^*$. The initial input for the KL-NMF problem is a one-step MU update of a Unif$(0,1)$ random matrix.}. It shows that SCI-PI outperforms the other 2 for all simulated data sets. Also, all methods seems to exhibit linear convergence.

\paragraph{KL-NMF on Real world data sets} Next, we test the 4 algorithms on the data sets in Table~\ref{table:1}. We estimate $k = 20$ factors. At each iteration, all 4 algorithms solve $m$ subproblems simultaneously for $W$ and then alternatively for $H$.

\begin{figure}[t]
    \centering
    \centerline{\includegraphics[width=5.5in]{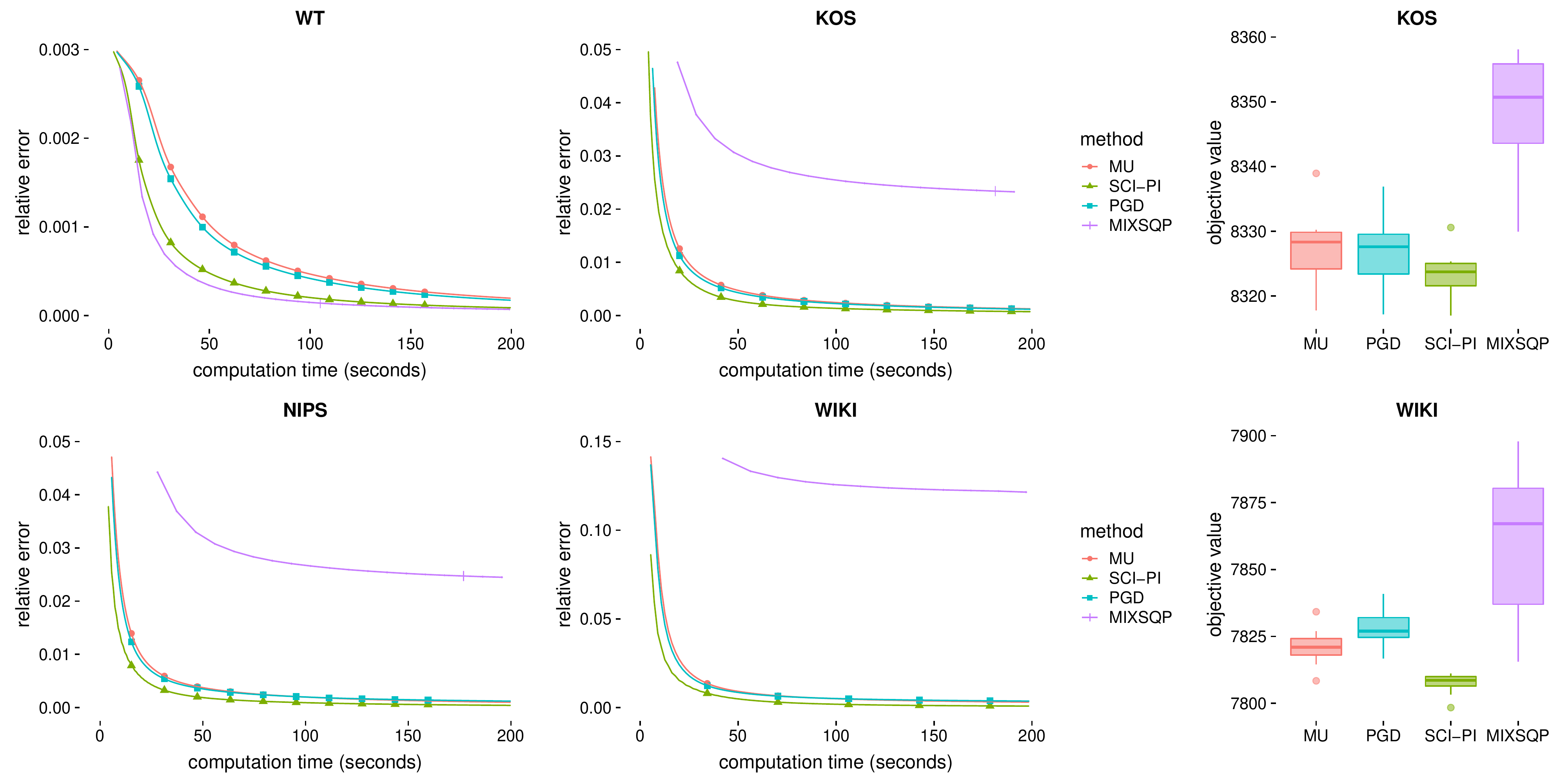}}
    \caption{{\it (Left and center)} Convergence of the 4 NMF algorithms. {\it (Right)} Boxplots containing 10 objective values achieved after 400 seconds.}
    \label{fig:nmf}
\end{figure}

The result is summarized in Figure~\ref{fig:nmf}\footnote{In all plots we do not show the first few iterations. The initial random solutions have the gap of approximately 50\% which drops to a few percent after 10 iterations where the plots start.}. The convergence plots are based on the average relative errors over 10 repeated runs with random initializations. The result shows that SCI-PI is an overall winner, showing faster convergence rates. The stopping criterion is the same as above. To assess the overall performance when initialized differently, we select KOS and WIKI and run MU, PGD, SCI-PI, and MIXSQP 10 times\footnoteref{note1}. The 3 algorithms except MIXSQP have (approximately) the same computational cost per iteration, take runtime of 391, 396, 408 seconds for KOS data and 372, 390, 418 seconds for WIKI data, respectively for 200 iterations. MIXSQP has a larger per iteration cost. After 400 seconds, SCI-PI achieves lowest objective values in all cases but one for each data set (38 out of 40 in total). Thus it clearly outperforms other methods and also achieves the lowest variance. Unlike the other 3 algorithms, SCI-PI is not an ascent algorithm but an eigenvalue-based fixed-point algorithm. We observe that sometimes SCI-PI converges to a better solution due to this fact. Admittedly, non-monotone convergence of SCI-PI can hurt reliability of the solution but for the KL-NMF problem its performance turns out to be stable.

\subsection{Gaussian Mixture Model and Independent Component Analysis}

In this subsection, we study the empirical performance of SCI-PI when it is applied to GMM and ICA.

\paragraph{GMM} \label{subsec:gmm}
\begin{figure}[h]
    \centering
    \centerline{\includegraphics[width=5.5in]{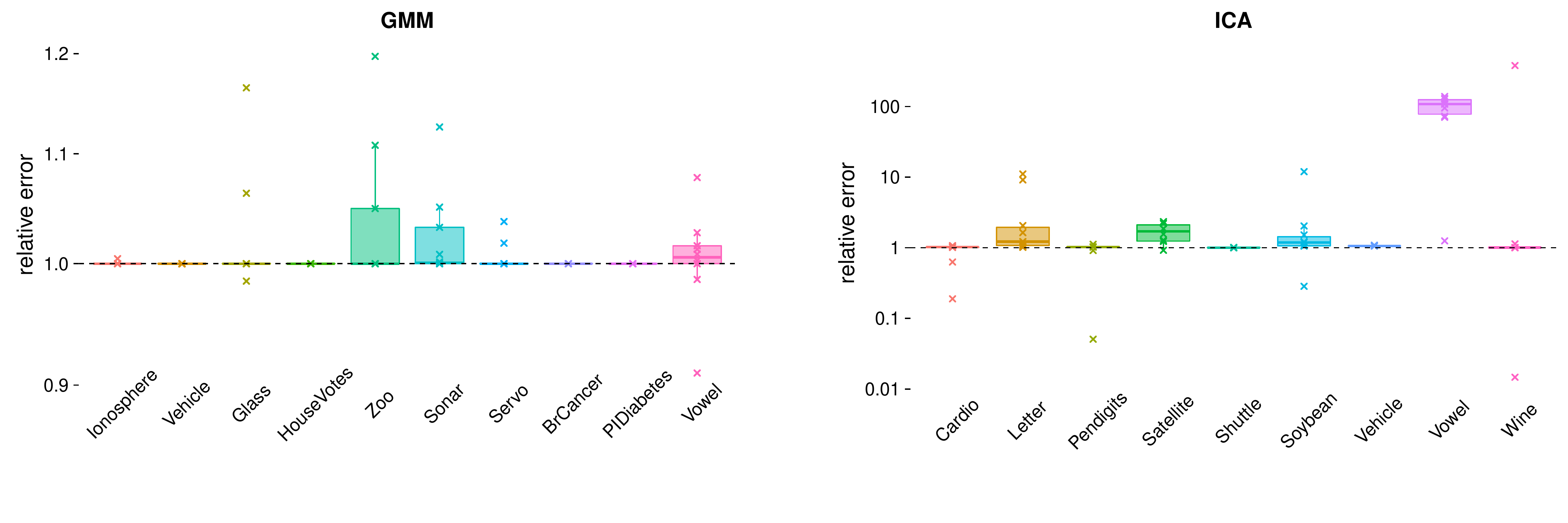}}
    \caption{The relative error $f_{\textrm{SCI-PI}}^* / f_{\textrm{EM}}^*$ for GMM {\it (Left)} and $f_{\textrm{SCI-PI}}^* / f_{\textrm{FastICA}}^*$ for ICA {\it (Right)}.}
    \label{fig:gmm}
\end{figure}

GMM fits a mixture of Gaussian distributions to the underlying data. Let $L_{ik} = \mathcal{N}(x_i;\mu_k,\Sigma_k)$ where $i$ is the sample index and $k$ the cluster index and let $\pi$ be the actual mixture proportion vector. GMM fits into our restricted scale invariant setting (Section~\ref{subsec:restricted}) with reparametrization, but the gradient update for $\mu_k,\Sigma_k$ is replaced by the exact coordinate ascent step. The EM and SCI-PI updates for $\pi$ can be written respectively as 
\baa
r = \mathds{1} \oslash (L\pi),\quad \pi_k^{\rm new} \propto \pi \odot (L^T r)\quad 
\textrm{(EM)},\quad \pi_k^{\rm new} \propto \pi \odot (\alpha + L^Tr)^{\odot 2} \quad \textrm{(SCI-PI).}
\eaa
We compare SCI-PI and EM for different real-world data sets from Table~\ref{table:2}. All the algorithms initialize from the same standard Gaussian random variable, repeatedly for 10 times. The result is summarized in the left panel in Figure~\ref{fig:gmm}. The stopping criterion is $\|x_{k+1} - x_k \| < 10^{-8}$. In some cases, SCI-PI achieves much larger objective values even if initialized the same. In many cases the 2 algorithms exhibit the same performance. This is because estimation of $\mu_k$'s and $\Sigma_k$'s are usually harder than estimation of $\pi$, and EM and SCI-PI have the same updates for $\mu$ and $\Sigma$. For a few cases EM outperforms SCI-PI. Let us mention that SCI-PI and EM have the same order of computational complexity and require 591 and 590 seconds of total computation time, respectively.
\paragraph{ICA} \label{subsec:ica}
We implement SCI-PI on the Kurtosis-based ICA problem \citep{hyvarinen2004independent} and compare it with the benchmark algorithm FastICA \citep{hyvarinen1999fast}, which is the most popular algorithm. Given a pre-processed\footnote{A centered matrix $\widetilde{W} = n^{1/2}UDV^T$ is pre-processed by $W = \widetilde{W} VD^{-1} V^T$ so that $W^TW = nVV^T$.} data matrix $W \in \mathbb{R}^{n\times d}$, we seek to maximize an approximated negative entropy $f(x) = \sum_{i=1}^n \left[(w_i^T x)^4 - 3\right]^2$ subject to $x \in \partial\mathcal{B}_d$, for maximizing Kurtosis-based non-Gaussianity \citep{hyvarinen2000independent}. This problem fits into the sum of scale invariant setting (Section~\ref{subsec:sumofscaleinv}). SCI-PI iterates $x_{k+1} \leftarrow W^T[(Wx_k)^{\odot 4} - 3\mathds{1}_n) \odot (Wx_k)^{\odot 3}]$ and FastICA iterates $x_{k+1} \leftarrow W^T (Wx_k)^{\odot 3} - 3(\mathds{1}^T (Wx_k)^{\odot 2})x_k$, both followed by normalization. 

In Figure~\ref{fig:gmm} (right panel), we compare SCI-PI and FastICA on the data sets in Table~\ref{table:3}. The majority of data points (81 out of 100 in total) show that SCI-PI tends to find a better solution with a larger objective value, but in a few cases SCI-PI converges to a sub-optimal point. Both algorithms are fixed-point based and thus have no guarantee of global convergence but overall SCI-PI outperforms FastICA. SCI-PI and FastICA have the same order of computational complexity and require 11 and 12 seconds of total computation time, respectively.

\section{Final Remarks}
In this paper, we propose a new class of optimization problems called the scale invariant problems, together with a generic solver SCI-PI, which is indeed an eigenvalue-based fixed-point iteration. We showed that SCI-PI directly generalizes power iteration and enjoys similar properties such as that SCI-PI has local linear convergence under mild conditions and its convergence rate is determined by eigenvalues of the Hessian matrix at a solution. Also, we extend scale invariant problems to problems with more general settings. We show by experiments that SCI-PI can be a competitive option for numerous important problems such as KL-NMF, GMM and ICA. Finding more examples and extending SCI-PI further to a more general setting is a promising direction for future studies.

\newpage
\onecolumn
\appendix
\section{Additional Lemmas}
On several occasions, we use if $x \in \partial B_d, \, y \in \partial B_d$, then 
\begin{equation*}
\| x - y\|^2 = \| x \|^2 + \|y \|^2 - 2x^Ty = 2(1-x^Ty).
\end{equation*}
Note that if $x^Ty \geq 0 $, then 
\begin{equation*}
\sqrt{1-(x^Ty)^2} = \sqrt{(1-x^Ty)(1+x^Ty)} \geq \sqrt{1-x^Ty} = \frac{\| x- y\|}{\sqrt{2}}.
\end{equation*}
By Cauchy-Schwarz, we also have
\begin{equation*}
\sqrt{1-(x^Ty)^2} = \sqrt{(1-x^Ty)(1+x^Ty)} \leq \sqrt{2} \sqrt{1-x^Ty} = \| x-y\|.
\end{equation*}
\subsection{For the Proofs of Theorem~\ref{thm:GPM-local-convergence} and Theorem~\ref{thm:SCI-PI-local-convergence-SSI}}
\begin{lemma}
\label{lemma:SCI-PI-local-convergence-small-delta}
Let $\{v_1,\ldots,v_d \}$ be an orthogonal basis in $\mathbb{R}^d$ with $x^* = v_1$ and $\{x_k\}_{k=0,1,\cdots}$ be the sequence of iterates generated by SCI-PI. If for every $x \in \partial \mathcal{B}_d$ we have
\begin{align}
\label{lemma1:condition}
\nabla f(x)^Tv_1 = \lambda^* + \alpha(x), \quad \sum_{i=2}^d (\nabla f(x)^T v_i)^2 \leq
\left(
\bar{\lambda}_2 \|x-x^*\| + \beta(x)
\right)^2
\end{align}
where 
\begin{equation*}
\alpha(x) = o(\sqrt{\|x-x^*\|}), \quad \beta(x) = o(\|x-x^*\|),
\end{equation*}
then there exists some $\delta > 0$ such that under the initial condition $1-x_0^Tx^* < \delta$, we have
\begin{align*}
1 - (x_k^Tx^*)^2
\leq 
\prod_{t=0}^{k-1} \bigg( \frac{\bar{\lambda}_2}{\lambda^*} + \gamma_t \bigg)^{2} \left( 1 - (x_0^Tx^*)^2 \right), \, \, \frac{\bar{\lambda}_2}{\lambda^*} + \gamma_t < 1, \, \, \textup{and} \, \, \lim_{k \rightarrow \infty} \gamma_k = 0.
\end{align*}
\end{lemma}
\begin{proof}
By \eqref{lemma1:condition} for every $x \in \partial \mathcal{B}_d$, we have
\begin{equation*}
\frac{\sum_{i=2}^d (\nabla f(x)^Tv_i)^2}{(\nabla f(x)^Tv_1)^2} \leq \left(
\frac{\bar{\lambda}_2 \|x-x^*\| + \beta(x)}{\lambda^* + \alpha(x)}
\right)^2.
\end{equation*}
Let
\begin{align*}
\frac{\bar{\lambda}_2 \|x-x^*\| + \beta(x)}{\lambda^* + \alpha (x)}
&= \frac{\bar{\lambda}_2}{\lambda^*} \|x-x^*\| + \theta(x).
\end{align*}
Then, we have $\theta(x) = o(\|x-x^*\|)$ and
\begin{align}
\label{lemma1:recurrence-relation-2}
\frac{\sum_{i=2}^d (\nabla f(x)^Tv_i)^2}{(\nabla f(x)^Tv_1)^2}
\leq
\left( 
\frac{\bar{\lambda}_2}{\lambda^*} + \frac{ \theta(x)}{\|x-x^*\|}  
\right)^2 \|x - x^* \|^2.
\end{align}
Letting
\begin{align}
\label{lemma1:limit-zero-terms}
{\epsilon(x)} = \frac{ \theta(x)}{\|x-x^*\|},
\end{align}
we can further represent \eqref{lemma1:recurrence-relation-2} as
\begin{align}
\frac{\sum_{i=2}^d (\nabla f(x)^Tv_i)^2}{(\nabla f(x)^Tv_1)^2} &\leq
\left( 
\frac{\bar{\lambda}_2}{\lambda^*} + {\epsilon}(x)
\right)^2 \left( 1 + \frac{1-x^Tx^*}{1+x^Tx^*} \right) \left( 1 - (x^Tx^*)^2 \right) \nonumber \\
&= \left( 
\frac{\bar{\lambda}_2}{\lambda^*} + \gamma(x) 
\right)^2
\left( 1 - (x^Tx^*)^2 \right) \label{lemma1:opt-bound-inequality}
\end{align}
where
\begin{align}
\label{lemma1:def-epsilon}
\gamma(x) = \frac{\bar{\lambda}_2}{\lambda^*} \left( \frac{1-x^Tx^*}{1+x^Tx^*+\sqrt{2(1+x^Tx^*)}} \right) + \epsilon(x) \sqrt{ 1 + \frac{1-x^Tx^*}{1+x^Tx^*}}.
\end{align}

From \eqref{lemma1:condition}, there exists some $\delta_1 > 0$ such that if $1-x^Tx^*<\delta_1$, then 
\begin{align}
\label{lemma1:delta1-bound}
\nabla f(x)^Tv_1 > 0.
\end{align}
Also, by \eqref{lemma1:limit-zero-terms}, for any $\bar{\gamma}>0$ satisfying
\begin{align}
\label{lemma1:ratio}
\frac{\bar{\lambda}_2}{\lambda^*} + \bar{\gamma} < 1,
\end{align}
there exists some constant $\delta_2 > 0$ such that if $1-x^Tx^*<\delta_2$, then
\begin{align}
\label{lemma1:delta2-bound}
|{\epsilon}(x)| \leq \frac{\bar{\gamma}}{4}.
\end{align}

Let $\delta = \min \{\delta_1, \delta_2, \frac{\lambda^*}{\bar{\lambda}_2} \bar{\gamma}, 1 \}$. Before proving the main result, we first 
show the following two statements:
\vspace{2mm}
\begin{enumerate}[wide, labelwidth=!, labelindent=0pt]
    \item If  $1 - x_k^Tx^* < \delta$, then we have
    \begin{equation}
    \label{lemma1:sub-result1}
    x_{k+1}^Tx^* > 0, \, \, 1 - (x_{k+1}^Tx^*)^2 \leq  \left( \frac{\bar{\lambda}_2}{\lambda^*} + {\gamma}_k \right)^2 \left( 1-(x_{k}^Tx^*)^2 \right), \, \,
    \textup{and} \, \, \gamma_k \leq \bar{\gamma}.
    \end{equation}
    Since $\delta < 1$, we have $x_k^Tx^* > 0$. 
    Also, from $1-x_k^Tx^* < \delta_1$ and $x^* = v_1$, using the update rule of SCI-PI and \eqref{lemma1:delta1-bound}, we obtain
    \begin{equation*}
    x_{k+1}^Tx^* = \frac{\nabla f(x_k)^Tx^*}{\|\nabla f(x_k) \|} = \frac{\nabla f(x_k)^Tv_1}{\|\nabla f(x_k) \|} > 0.
    \end{equation*}
    On other the hand, since $|x_{k+1}^Tv_1| \leq \|x_{k+1}\| \|v_{1}\|=1$, we have
    \begin{align*}
    1 - (x_{k+1}^Tx^*)^2 \leq \frac{1 - (x_{k+1}^Tv_1)^2}{(x_{k+1}^Tv_1)^2}.
    \end{align*}
    Also, from the fact that $\{v_1,\ldots,v_d \}$ forms an orthogonal basis in $\mathbb{R}^d$, we have $\nabla f(x_k) = \sum_{i=1}^d (\nabla f(x_k)^Tv_i) v_i$ and $\| \nabla f(x_k) \|^2 = \sum_{i=1}^d (\nabla f(x_k)^Tv_i)^2$. Using the update rule of SCI-PI, we have
    \begin{align*}
    \frac{1 - (x_{k+1}^Tv_1)^2}{(x_{k+1}^Tv_1)^2} 
    = \frac{\|\nabla f(x_k) \|^2- (\nabla f(x_{k})^Tv_1)^2}{(\nabla f(x_{k})^Tv_1)^2}
    = \frac{\sum_{i=2}^d (\nabla f(x_k)^Tv_i)^2}{(\nabla f(x_{k})^Tv_1)^2},
    \end{align*}
    resulting in
    \begin{align*}
    1 - (x_{k+1}^Tx^*)^2 \leq \frac{\sum_{i=2}^d (\nabla f(x_k)^Tv_i)^2}{(\nabla f(x_{k})^Tv_1)^2}.
    \end{align*}
    Let $\gamma_k = \gamma(x_k)$ and $\epsilon_k = \epsilon(x_k)$. Since $x_k^Tx^* > 0$ and $1-x_k^Tx^* < \min \{\delta_2, \frac{\lambda^*}{\bar{\lambda}_2} \bar{\gamma} \}$, from \eqref{lemma1:def-epsilon}, we have
    \begin{align*}
    \gamma_k = \frac{\bar{\lambda}_2}{\lambda^*} \left( \frac{1-x_k^Tx^*}{1+x_k^Tx^*+\sqrt{2(1+x_k^Tx^*)}} \right) + {\epsilon}_k \sqrt{ 1 + \frac{1-x_k^Tx^*}{1+x_k^Tx^*}} \leq \frac{\bar{\gamma}}{2} + \frac{\bar{\gamma}}{2} = \bar{\gamma},
    \end{align*}

    \item Using mathematical induction, we show that if 
    \begin{equation}
    \label{lemma1:induction-initial}    
    1 - x_0^Tx^* < \delta,
    \end{equation}
    then, for all $k \geq 0$. we have
    \begin{equation}
    \label{lemma1:induction-main}
    1 - x_k^Tx^* < \delta.
    \end{equation}
    By \eqref{lemma1:induction-initial}, we have $1-x_0^Tx^* < \delta$, which shows the base case.
    Next, suppose that $1 - x_k^Tx^* < \delta$ holds. Then, we have \eqref{lemma1:sub-result1}. Also, from $\delta < 1$, we have $x_k^Tx^* > 0$. Since
    \begin{equation*}
    x_{k+1}^Tx^* > 0, \, \, x_k^Tx^* > 0, \, \, 1-(x_{k+1}^Tx^*)^2 \leq \left( \frac{\bar{\lambda}_2}{\lambda^*} + \bar{\gamma} \right)^2 \left( 1-(x_k^Tx^*)^2 \right) < 1-(x_k^Tx^*)^2
    \end{equation*}
    we have 
    \begin{align*}
    1 - x_{k+1}^Tx^* < 1 - x_{k}^Tx^* < \delta,
    \end{align*}
    which completes the induction proof.
\end{enumerate}
\vspace{2mm}

Now, we prove the main statement. Since \eqref{lemma1:induction-main} holds for all $k \geq 0$, we can repeatedly apply \eqref{lemma1:sub-result1} to obtain
\begin{equation*}
1 - (x_{k}^Tx^*)^2 
\leq 
\prod_{t=0}^{k-1} \bigg( \frac{\bar{\lambda}_2}{\lambda^*} + \gamma_t \bigg)^{2} \left( 1 - (x_0^Tx^*)^2 \right),
\, \, \textup{and} \quad 
\frac{\bar{\lambda}_2}{\lambda^*} + {\gamma}_k \leq \frac{\bar{\lambda}_2}{\lambda^*} + \bar{\gamma} \leq 1.
\end{equation*}
Since 
\begin{align}
\label{lemma1:sub-result2}
1 - (x_{k}^Tx^*)^2 
< 
\left( \frac{\bar{\lambda}_2}{\lambda^*} + \bar{\gamma} \right)^{2k} \left( 1-(x_{0}^Tx^*)^2 \right),
\end{align}
we have $(x_k^Tx^*)^2 \rightarrow 1$. Moreover, from that $x_k^Tx^* > 0$ for all $k\geq 0$ by \eqref{lemma1:induction-main}, we have $x_k \rightarrow x^*$, and thus $\lim_{k \rightarrow \infty} \gamma_k = 0$ by \eqref{lemma1:def-epsilon}. With \eqref{lemma1:sub-result2}, this gives the desired result.
\end{proof}
\begin{lemma}
\label{lemma:SCI-PI-local-convergence-explicit-delta}
Let $\{v_1,\ldots,v_d \}$ be an orthogonal basis in $\mathbb{R}^d$. If $x^* = v_1$ and a sequence of iterates $\{x_k\}_{k=0,1,\cdots}$ generated by SCI-PI satisfies
\begin{align}
\label{lemma2:condition-v1}
\nabla f(x_k)^Tv_1 \geq A - B (1-x_k^Tx^*) - C \sqrt{1-x_k^Tx^*}
\end{align}
and
\begin{align}
\label{lemma2:condition-vi}
\sum_{i=2}^d (\nabla f(x_k)^T v_i)^2 \leq \left( D \sqrt{1-(x_k^Tx^*)^2} + E \sqrt{2(1-x_k^Tx^*)} + \frac{F}{2} \|x_k-x^*\|^2 \right)^2
\end{align}
where $A>0$ and $B,C,D,E,F$ are non-negative real numbers such that 
\begin{equation*}
B+C>0, \quad \frac{D+E}{A} < 1.
\end{equation*}
Then, under the initial condition that $1- x_0^Tx^* < \delta$ where
\begin{align}
\delta = \min \left \{ \left(\frac{A}{B+C}\right)^2, \left(\frac{A-D-E}{B+C+E+F}\right)^2 , 1 \right \},
\end{align}
we have
\begin{align*}
1 - (x_k^Tx^*)^2
\leq 
\prod_{t=0}^{k-1} \bigg( \frac{D+E}{A} + \gamma_t \bigg)^{2} \left( 1 - (x_0^Tx^*)^2 \right), \frac{D+E}{A} + \gamma_t < 1,  \, \textup{and} \lim_{k \rightarrow \infty} \gamma_k = 0.
\end{align*}
\end{lemma}
\begin{proof}
In order to prove the main result, we first 
show the following two statements:
\begin{enumerate}[wide, labelwidth=!, labelindent=2pt]
    \item If $1-x_k^Tx^* < \delta$, then we have
    \begin{align}
    \label{lemma2:sub-result}
    x_{k+1}^Tx^* > 0, 1-(x_{k+1}^Tx^*)^2 < \left( \frac{D+E}{A} + \gamma_{k} \right)^2 \left( 1-(x_{k}^Tx^*)^2 \right), \frac{D+E}{A} + \gamma_{k} < 1
    \end{align}
    for all $k \geq 0$ where 
    \begin{align}
    \label{lemma2:def-epsilon}
    \gamma_k = \frac{\big( A(E+F)+(B+C)(D+E) \big) \sqrt{1-x_k^Tx^*}}{A \left( A-(B+C)\sqrt{1-x_k^Tx^*}\right)}.
    \end{align}
    Since $0< x_k^Tx^* \leq 1$, we have $\sqrt{1-x_k^Tx^*} \geq 1-x_k^Tx^*$. Using $x^*= v_1$, the update rule of SCI-PI, \eqref{lemma2:condition-v1}, and the fact that $\delta \leq (A/(B+C))^2$, we have
    \begin{align}
    \label{lemma2:x-k-1-positive}
    x_{k+1}^Tx^* = \frac{\nabla f(x_k)^Tv_1}{\| \nabla f(x_k) \|} \geq \frac{A-B(1-x_k^Tx^*)-C\sqrt{1-x_k^Tx^*}}{\| \nabla f(x_k) \|} > 0
    \end{align}
    since
    \begin{equation*}
        \frac{A-B(1-x_k^Tx^*)-C\sqrt{1-x_k^Tx^*}}{\| \nabla f(x_k) \|} \geq \frac{A-(B+C) \sqrt{1-x_k^Tx^*}}{\| \nabla f(x_k) \|} > 0.
    \end{equation*}
    Using the same arguments in Lemma~\ref{lemma:SCI-PI-local-convergence-small-delta}, we have 
    \begin{align}\label{lemma2:rewrite}
    1 - (x_{k+1}^Tx^*)^2 \leq \frac{\sum_{i=2}^d (\nabla f(x_k)^Tv_i)^2}{(\nabla f(x_{k})^Tv_1)^2}.
    \end{align}
    By \eqref{lemma2:x-k-1-positive}, we have
    \begin{align*}
    A - B (1-x_k^Tx^*) - C \sqrt{1-x_k^Tx^*} > 0.    
    \end{align*}
    Therefore, by plugging \eqref{lemma2:condition-v1} and \eqref{lemma2:condition-vi} into \eqref{lemma2:rewrite} and using that $x_k^Tx^* > 0$, we have
    \begin{align}
    1 - (x_{k+1}^Tx^*)^2
    &\leq \left( \frac{D \sqrt{1-(x_k^Tx^*)^2} + E \sqrt{2(1-x_k^Tx^*)} + \frac{F}{2} \|x_k-x^*\|^2}{A - B (1-x_k^Tx^*) - C \sqrt{1-x_k^Tx^*}} \right)^2 \nonumber \\
    &= \left( \frac{D + E \sqrt{1+\frac{1-x_k^Tx^*}{1+x_k^Tx^*}}+ F \sqrt{\frac{1-x_k^Tx^*}{1+x_k^Tx^*}}}{A - B (1-x_k^Tx^*) - C \sqrt{1-x_k^Tx^*}} \right)^2 \left( 1 - (x_k^Tx^*)^2 \right) \nonumber \\
    &\leq \left( \frac{D + E \left( 1+\sqrt{1-x_k^Tx^*} \right)+ F \sqrt{{1-x_k^Tx^*}}}{A - (B+C) \sqrt{1-x_k^Tx^*}} \right)^2 \left( 1 - (x_k^Tx^*)^2 \right) \nonumber \\
    &= \left( \frac{D+E}{A} + \gamma_k \right)^2 \left( 1 - (x_k^Tx^*)^2 \right)
    \label{eq:GPM-opt-gap}
    \end{align}
    where we use the fact that $\sqrt{1+x} \leq 1+\sqrt{x}$ for $x \geq 0$ to derive the second inequality. Lastly, from
    $$
    \sqrt{1-x_k^Tx^*} < \sqrt{\delta} \leq  \frac{A-D-E}{B+C+E+F},
    $$
    we have
    \begin{align*}
    \gamma_k < 1 - \frac{D+E}{A}.
    \end{align*}
    \item Using mathematical induction, we show that if
    \begin{align}
    \label{lemma2:initial}
    1-x_0^Tx^* < \delta,
    \end{align}
    then, for all $k \geq 0$, we have
    \begin{align}
    \label{lemma2:opt-gap}
    1 - x_{k}^Tx^* < \delta.
    \end{align}
    By \eqref{lemma2:initial}, we have $1-x_0^Tx^* < \delta$, which proves the base case. Next, suppose that we have $1-x_k^Tx^* < \delta$. Then, we have \eqref{lemma2:sub-result}. Also, from $\delta < 1$, we have $x_k^Tx^*>0$. Since
    \begin{equation*}
    x_{k+1}^Tx^* > 0, \, \, x_k^Tx^* > 0, \, \, 1-(x_{k+1}^Tx^*)^2 < 1-(x_k^Tx^*)^2,
    \end{equation*}
    we have 
    \begin{align*}
    1 - x_{k+1}^Tx^* < 1 - x_{k}^Tx^* < \delta.
    \end{align*}
    This completes the induction proof.
\end{enumerate}
Now, we prove the main statement. Since \eqref{lemma2:opt-gap} holds for all $k \geq 0$, by repeatedly applying \eqref{lemma2:sub-result}, we obtain
\begin{align}
\label{lemma2:main-recurrence}
1 - (x_k^Tx^*)^2
\leq 
\prod_{t=0}^{k-1} \bigg( \frac{D+E}{A} + \gamma_t \bigg)^{2} \left( 1 - (x_0^Tx^*)^2 \right), \, \, \textup{and} \, \, \frac{D+E}{A} + \gamma_k < 1.
\end{align}
Since $({D+E})/{A} + \gamma_k < 1$ for all $k \geq 0$, $1-(x_k^Tx^*)^2$ is monotone decreasing, and so is $1-x_k^Tx^*$ by non-negativity. Moreover, from that $\gamma_k$ is a monotone increasing function of $1-x_k^Tx^*$, we have $\gamma_{k+1} \leq \gamma_k$ for all $k \geq 0$, resulting in
\begin{align*}
\prod_{t=0}^{k-1} \bigg( \frac{D+E}{A} + \gamma_t \bigg)^{2} \leq \left( \frac{D+E}{A} + \gamma_0 \right)^{2k}.
\end{align*}
Since $({D+E})/{A} + \gamma_0 < 1$ by \eqref{lemma2:sub-result}, we have $(x_k^Tx^*)^2 \rightarrow 1$. Due to $x_k^Tx^* > 0$ for all $k\geq 0$, this implies $x_k \rightarrow x^*$, and thus $\lim_{k \rightarrow \infty} \gamma_k = 0$ due to \eqref{lemma2:def-epsilon}. With \eqref{lemma2:main-recurrence}, this gives the desired result.
\end{proof}

\subsection{For the Proofs of Theorem~\ref{thm:joint-GPM-local-convergence} and Theorem~\ref{thm:restricted-GPM-local-convergence}}
\begin{lemma}
\label{lemma:GPM-x}
Suppose that $f(w,z)$ is scale invariant in $w \in \mathbb{R}^{d_w}$ for each $z \in \mathbb{R}^{d_z}$ and twice continuously differentiable on an open set containing $\partial \mathcal{B}_{d_w} \times \partial \mathcal{B}_{d_z}$. Let $(w^*,z^*)$ be a point satisfying
\begin{equation*}
\nabla_w f(w^*,z^*) = \lambda_w^* w^*, \quad \lambda_w^* > \bar{\lambda}_{2}^w = {\textstyle \max_{2\leq i \leq d_w}} |\lambda_{i}^w|, \quad w^* = v_1^w
\end{equation*}
where $(\lambda^w_i,v^w_i)$ is an eigen-pair of $\nabla_{ww}^2 f(w^*,z^*)$. Then, for any $w \in \partial \mathcal{B}_{d_w}$ and $z \in \partial \mathcal{B}_{d_z}$, 
we have
\begin{align*}
\nabla_w f(w,z)^T v_1^w = \lambda^*_w + (z-z^*)^T \nabla_{zw}^2 f(w^*, z^*) w^* 
+ \alpha^w(w,z)
\end{align*}
and
\begin{align*}
\sum_{i=2}^{d_w} (\nabla_w f(w,z)^T v_i^w)^2 &\leq \left( \bar{\lambda}_{2}^w \sqrt{1-(w^Tw^*)^2} + \nu^{wz} \|z-z^*\| + \beta^w(w,z) 
\right)^2
\end{align*}
where
\begin{align*}
\alpha^w(w,z) = o \left( \left \|
\begin{bmatrix}
w - w^* \\
z - z^*
\end{bmatrix}
\right \| \right)
, \quad
\beta^w(w,z) = o \left( \left \|
\begin{bmatrix}
w - w^* \\
z - z^*
\end{bmatrix}
\right \| \right).
\end{align*}
Therefore, we have
\begin{align*}
1- \frac{(\nabla_w f(w,z)^Tw^*)^2}{\|\nabla_w f(w,z) \|^2} \leq 
\left(
\frac{\bar{\lambda}_2^w}{\lambda^*_w} \sqrt{1-(w^Tw^*)^2} + \frac{\nu^{wz}}{\lambda^*_w} \|z-z^*\| + \theta^w(w,z)
\right)^2
\end{align*}
where
\begin{equation*}
\nu^{wz} = \| \nabla_{wz}^2 f(w^*,z^*) \|, \quad
\theta^w(w,z) = o \left( \left \|
\begin{bmatrix}
w - w^* \\
z - z^*
\end{bmatrix}
\right \| \right).
\end{equation*}
\end{lemma}
\begin{proof}
Since $\nabla_{ww}^2 f(w^*,z^*)$ is real and symmetric, without loss of generality, we assume that $\{v_1^w,\ldots,v_{d_w}^w\}$ forms an orthogonal basis in $\mathbb{R}^{d_w}$.

By Taylor expansion of $\nabla_w f (w,z)^Tv_i^w$ at $(w^*,z^*)$, we have
\begin{align*}
\nabla_w f(w, z)^T v_i^w = \nabla_x f(w^*, z^*)^T v_i^w + 
\begin{bmatrix}
w - w^* \\
z - z^*
\end{bmatrix}
^T
\begin{bmatrix}
\nabla_{ww}^2 f(w^*,z^*) \\
\nabla_{zw}^2 f(w^*,z^*)
\end{bmatrix}
v_i^w + R_i^w(w,z)
\end{align*}
where
\begin{equation*}
R_i^w(w,z) = 
o \left( \left \|
\begin{bmatrix}
w - w^* \\
z - z^*
\end{bmatrix}
\right \| \right).
\end{equation*}
Using $\nabla_w f(w^*,z^*) = \lambda^*_w w^*$ and $w^* = v_1^w$, we have
\begin{align*}
\nabla_w f(w^*,z^*)^T v_1^w = \lambda^*_w, \quad (w-w^*)^T \nabla_{ww}^2 f(w^*,z^*) v_1^w = - \lambda_1^w ( 1-w_k^Tw^*).
\end{align*}
Therefore, we obtain
\begin{align}
\label{lemma:restricted-grad-x-k-v-1}
\nabla_w f(w,z)^T v_1^w = \lambda^*_w + (w-w^*)^T \nabla_{zw}^2 f(w^*, z^*) w^* + \alpha^w(w,z)
\end{align}
where
\begin{equation*}
\alpha^w(w,z) = R_1^w(w,z) - \lambda_1^w ( 1-w^Tw^*) = o \left( \left \|
\begin{bmatrix}
w - w^* \\
z - z^*
\end{bmatrix}
\right \| \right).
\end{equation*}

In the same way, for $2 \leq i \leq d_w$, we have
\begin{equation*}
\nabla_w f(w^*,z^*)^T v_i^w = \lambda^*_w (w^*)^Tv_i^w = 0, \quad (w-w^*)^T \nabla_{ww}^2 f(w^*,z^*) v_i^w = \lambda_i^w w^Tv_i^w,
\end{equation*}
resulting in
\begin{align}
\label{lemma:restricted-grad-x-k-v-i}
\nabla_w f(w,z)^T v_i^w = \lambda_i^w w^Tv_i^w + (z-z^*)^T \nabla_{zw}^2 f(w^*, z^*) v_i^w + R_i^w(w,z).
\end{align}
From \eqref{lemma:restricted-grad-x-k-v-i}, we obtain
\begin{align*}
\sum_{i=2}^{d_w} (\nabla_w f(w,z)^T v_i^w)^2 
& = 
\sum_{i=2}^{d_w} (\lambda_i^w)^2 (w^Tv_i^w)^2 + \sum_{i=2}^{d_w} \left( (z-z^*)^T \nabla_{zw}^2 f(w^*, z^*) v_i^w \right)^2 \\
& \, \, + \sum_{i=2}^{d_w} (R_i^w(w,z))^2 + 2 \sum_{i=2}^{d_w} \lambda_i^w (w^Tv_i^w) (z-z^*)^T \nabla_{zw}^2 f(w^*, z^*) v_i^w  \\
& \, \, + 2 \sum_{i=2}^{d_w} \lambda_i^w (w^Tv_i^w) R_i^w(w,z) \\
& \, \, + 2 \sum_{i=2}^{d_w} (z-z^*)^T \nabla_{zw}^2 f(w^*, z^*) v_i^w R_i^w(w,z).
\end{align*}
Since $\{v_1^w,\ldots,v_{d_w}^w\}$ forms an orthogonal basis in $\mathbb{R}^{d_w}$, with $w^* = v_1^w$ and $\|w\|^2=1$, we have
\begin{align*}
\sum_{i=2}^{d_w} (\lambda_i^w)^2 (w^Tv_i^w)^2 \leq (\bar{\lambda}_2^w)^2 \left( 1 - (w^Tw^*)^2 \right)
\end{align*}
and
\begin{align*}
\sum_{i=2}^{d_w} \left( (z-z^*)^T \nabla_{zw}^2 f(w^*, z^*) v_i^w \right)^2 \leq 
\| (z-z^*)^T \nabla_{zw}^2 f(w^*, z^*) \|^2 
\leq 
(\nu^{wz})^2 \|z-z^*\|^2.
\end{align*}
Let $\bar{R}_2^w (w,z) = {\textstyle \max_{2 \leq i \leq d_w} |R_i^w(w,z)|}$. 
Note that 
\begin{equation*}
\bar{R}_2^w (w,z) = o \left( \left \|
\begin{bmatrix}
w - w^* \\
z - z^*
\end{bmatrix}
\right \| \right).
\end{equation*}
Using the Cauchy-Shwartz inequality, we have
\begin{align*}
\sum_{i=2}^{d_w} \lambda_i^w (w^Tv_i^w) (z-z^*)^T \nabla_{zw}^2 f(w^*, z^*) v_i^w \leq \bar{\lambda}_2^w \nu^{wz} \| z - z^* \| \sqrt{1-(w^Tw^*)^2}.
\end{align*}
Also, we have
\begin{align*}
\sum_{i=2}^{d_w} \lambda_i^w (w^Tv_i^w) R_i^w(w,z)
\leq \bar{\lambda}_2^w \bar{R}_2^w (w,z) \sqrt{d_w} \sqrt{1-(w^Tw^*)^2}
\end{align*}
and
\begin{align*}
\sum_{i=2}^{d_w} R_i^w(w,z) (z - z^*)^T \nabla_{zw}^2 f(w^*,z^*) v_i^w \leq  \nu^{wz} \bar{R}_2^w (w,z) \sqrt{d_w}  \|z-z^*\|.
\end{align*}
Therefore, we obtain
\begin{align}
\label{lemma:restricted-grad-x-v-i-square-sum}
\sum_{i=2}^{d_w} (\nabla_w f(w,z)^T v_i^w)^2 &\leq \left( \bar{\lambda}_2^w \sqrt{1-(w^Tw^*)^2} + \nu^{wz} \|z-z^*\| + \beta^w(w,z)
\right)^2
\end{align}
where
\begin{equation*}
\beta^w(w,z) = \bar{R}_2^w (w,z) \sqrt{d_w}  
= o \left( \left \|
\begin{bmatrix}
x_k - x^* \\
y_k - y^*
\end{bmatrix}
\right \| \right).
\end{equation*}

Since $\{v_1^w,\ldots,v_{d_w}^w\}$ forms an orthogonal basis in $\mathbb{R}^{d_w}$ and $|w^T w^*|\leq \|w\|\|w^*\|=1$, we have
\begin{equation*}
1- \frac{(\nabla_w f(w,z)^Tw^*)^2}{\|\nabla_w f(w,z) \|^2}
\leq 
\frac{\sum_{i=2}^{d_w} (\nabla_w f(w,z)^Tv_i^w)^2}{(\nabla_w f(w,z)^Tv_1^w)^2}.
\end{equation*}
Using \eqref{lemma:restricted-grad-x-k-v-1} and \eqref{lemma:restricted-grad-x-v-i-square-sum}, we have
\begin{align*}
\frac{\sum_{i=2}^{d_w} (\nabla_w f(w,z)^Tv_i^w)^2}{(\nabla_w f(w,z)^Tv_1^w)^2}
&\leq 
\left(
\frac{\bar{\lambda}_2^w}{\lambda^*_w}  \sqrt{1-(w^Tw^*)^2} + \frac{\nu^{wz}}{\lambda^*_w} \| z-z^*\| + \theta^w(w,z)
\right)^2
\end{align*}
where
\begin{equation*}
\begin{aligned}
\theta^w(w,z) &= \frac{\beta^w (w,z)}{\lambda^*_w} - 
\left(
\frac{
\bar{\lambda}_2^w \sqrt{1-(w^Tw^*)^2} + \nu^{wz} \|z-z^*\| + \sqrt{d_w} \beta^w (w,z)}{\lambda^*_w} \right) \\
&\quad \cdot
\left(
\frac{(z-z^*)^T \nabla_{zw}^2 f(w^*, z^*) w^* 
+ \beta^w (w,z)}
{\lambda^*_w + (z-z^*)^T \nabla_{zw}^2 f(w^*, z^*) w^* + \beta^w (w,z)}
\right).
\end{aligned}
\end{equation*}
Since
\begin{equation*}
|(z-z^*)^T \nabla_{zw}^2 f(w^*, z^*) w^*| \leq \nu^{wz} \|z-z^*\|,
\end{equation*}
we have
\begin{align*}
|(z-z^*)^T \nabla_{zw}^2 f(w^*, z^*) w^*| \sqrt{1-(w^Tw^*)^2} \leq \frac{1}{2} \left( 1-(w^Tw^*)^2 \right) + \frac{1}{2} (\nu^{wz})^2 \| z-z^*\|^2
\end{align*}
and
\begin{align*}
\nu^{wz} |(z-z^*)^T \nabla_{zw}^2 f(w^*, z^*) w^*| \|z-z^*\| \leq (\nu^{wz})^2 \|z-z^*\|^2.
\end{align*}
From
\begin{align*}
1-(w^Tw^*)^2 = o \left( \left \|
\begin{bmatrix}
w - w^* \\
z - z^*
\end{bmatrix}
\right \| \right),
\quad \|z-z^*\|^2 = o \left( \left \|
\begin{bmatrix}
w - w^* \\
z - z^*
\end{bmatrix}
\right \| \right),
\end{align*}
we finally obtain
\begin{equation*}
\theta^w(w,z) = o \left( \left \|
\begin{bmatrix}
w - w^* \\
z - z^*
\end{bmatrix}
\right \| \right).
\end{equation*}
This completes the proof.
\end{proof}

\begin{lemma}
\label{lemma:GPM-y}
Suppose that $f(w,z)$ is $\mu$-strongly concave in $z \in \mathbb{R}^{d_z}$ with an $L$-Lipschitz continuous $\nabla_z f(w,z)$ for each $w \in \partial \mathcal{B}_{d_w}$ and three-times continously differentiable with respect to $x$ and $y$ on an open set containing $\partial \mathcal{B}_{d_w}$ and $\mathbb{R}^{d_z}$, respectively. Let $(w^*,z^*)$ be a point such that
$\nabla_z f(w^*,z^*) = 0$. Then, for any $w \in \partial \mathcal{B}_{d_w}$ and $z \in \partial \mathcal{B}_{d_z}$, with $\alpha = 2/(L+\mu)$, we have
\begin{align}
\label{proof:restricted-recurrence-y}
\| z + \alpha \nabla_z f(w,z) - z^*\| \leq \left( \frac{2\nu^{zw}}{L+\mu} \right) \|w-w^*\| + \left( \frac{L-\mu}{L+\mu} \right) \| z - z^* \| + \theta^z (w,z)
\end{align}
where
\begin{equation*}
\nu^{zw} = \| \nabla_{zw}^2 f(w^*,z^*) \|, \quad
\theta^z (w,z) = o \left( \left \|
\begin{bmatrix}
w - w^* \\
z - z^*
\end{bmatrix}
\right \| \right).
\end{equation*}
\end{lemma}
\begin{proof}
Let $\nabla_{z,i} f$ be the $i^{th}$ coordinate of $\nabla_z f$ and 
\begin{equation*}
H_{z,i} = 
\begin{bmatrix}
H_{z,i}^{ww} & H_{z,i}^{wz} \\
H_{z,i}^{zw} & H_{z,i}^{zz} 
\end{bmatrix}
\end{equation*}
be the Hessian of $\nabla_{z,i} f$. By Taylor expansion of $\nabla_{z,i} f (w,z)$ at $(w^*,z)$, we have
\begin{align}
\label{proof:restricted-grad-y-f-x-k-y-k}
\nabla_{z,i} f (w, z) = \nabla_{z,i} f(w^*,z) + \nabla_{zw,\cdot i}^2 f(w^*,z)^T (w-w^*) + R_i^z(w,z)
\end{align}
where $\nabla_{zw,\cdot i}^2 f(w^*,z) = \nabla_w \nabla_{z,i} f(w^*,z)$ denotes the $i^{th}$ column of $\nabla_{zw}^2 f(w^*,z)$ and
\begin{equation}
R_i^z(w,z) = \frac{1}{2} (w-w^*)^T H_{z,i}^{ww}(\hat{w}^i,z) (w-w^*), \quad \hat{w}^i \in \mathcal{N}(w,w^*).
\label{proof:restricted-def-R-i-z}
\end{equation}
Also, from $f$ being three-times continuously differentiable, we have
\begin{align}
\label{proof:restricted-hessian-x-star-y-k}
\nabla_{zw,\cdot i}^2 f(w^*,z) = \nabla_{zw,\cdot i}^2 f(w^*,z^*) + H_{z,i}^{wz} (w^*,\hat{z}^i) (z - z^*), \quad \hat{z}^i \in \mathcal{N}(z,z^*).
\end{align}
Since
\begin{align*}
|(z-z^*)^T H_{z,i}^{zw}(w^*,\hat{z}^i) (w-w^*)| &\leq \|H_{z,i}^{zw}(w^*,\hat{z}^i)\| \|w-w^*\| \|z-z^*\| \\
&\leq \frac{1}{2} \|H_{z,i}^{zw}(w^*,\hat{z}^i)\| \left( \|w-w^*\|^2 + \|z-z^*\|^2 \right),
\end{align*}
we have
\begin{align}
\label{proof:rextricted-cross-term-bound}
(z-z^*)^T H_{z,i}^{wz}(w^*,\hat{z}^i) (w-w^*) = o \left( \left \|
\begin{bmatrix}
w - w^* \\
z - z^*
\end{bmatrix}
\right \| \right).
\end{align}
By \eqref{proof:restricted-grad-y-f-x-k-y-k}, \eqref{proof:restricted-def-R-i-z},  \eqref{proof:restricted-hessian-x-star-y-k}, and \eqref{proof:rextricted-cross-term-bound}, we have
\begin{equation}
\label{proof:restricted-grad-y-final}
\begin{aligned}
\nabla_z f (w, z) &= \nabla_z f(w^*,z) + \nabla_{zw}^2 f(w^*,z^*) (w - w^*) + \bar{R}^z(w,z)
\end{aligned}
\end{equation}
where
\begin{equation*}
\bar{R}_i^z(w,z) = R_i^z(w,z) + (z-z^*)^T H_{z,i}^{zw}(w^*,\hat{z}^i) (w-w^*)
=
o \left( \left \|
\begin{bmatrix}
w - w^* \\
z - z^*
\end{bmatrix}
\right \| \right).
\end{equation*}

Using \eqref{proof:restricted-grad-y-final}, we have
\begin{align*}
z + \alpha \nabla_z f (w, z) - z^* &= z - z^* + \alpha \nabla_z f(w^*,z) + \alpha \nabla_{zw}^2 f(w^*,z^*) (w-w^*) + \bar{R}^z(w,z),
\end{align*}
resulting in
\begin{equation}
\label{proof:restricted-y-update-rule}
\begin{aligned}
\| z + \alpha \nabla_z f (w, z) - z^* \| & \leq \| z - z^* + \alpha \nabla_z f(w^*, z) \| \\
& \, \, + \alpha \| \nabla_{zw}^2 f(w^*,z^*) (w-w^*) \| + \| \bar{R}^z (w,z) \|.
\end{aligned}
\end{equation}
Since $-f(w^*,z)$ is $\mu$-strongly convex in $z$ with an $L$-Lipschitz continuous gradient $-\nabla_z f(w^*,z)$, by theory of convex optimization \citep[p.~270]{bubeck2015convex}, we have
\begin{align}
\label{proof:GA-y-k-bound-term-1}
\| z - z^* + \alpha \nabla_z f(w^*, z) \| \leq \left( \frac{L-\mu}{L+\mu} \right) \| z - z^* \|
\end{align}
due to $\alpha = {2}/(L+\mu)$. Also, we have
\begin{align}
\label{proof:GA-y-k-bound-term-2}
\alpha \| \nabla_{zw}^2 f(w^*,z^*) (w-w^*) \| \leq \left( \frac{2\nu^{zw}}{L+\mu} \right) \|w - w^* \|.
\end{align}

Plugging \eqref{proof:GA-y-k-bound-term-1}, \eqref{proof:GA-y-k-bound-term-2} into \eqref{proof:restricted-y-update-rule}, we finally obtain
\begin{align*}
\| z - z^* + \alpha \nabla_z f(w^*, z) \| \leq \left( \frac{L-\mu}{L+\mu} \right) \| z - z^* \| + \left( \frac{2\nu^{zw}}{L+\mu} \right) \|w-w^*\| + \theta^z(w,z)
\end{align*}
where
\begin{equation*}
\theta^z(w,z)
=
\| \bar{R}^z (w,z) \|
=
o \left( \left \|
\begin{bmatrix}
w - w^* \\
z - z^*
\end{bmatrix}
\right \| \right).
\end{equation*}
\end{proof}


\begin{lemma}
\label{lemma:asymptotic-spectral-property}
Let $M$ be a $2 \times 2$ matrix such that
\begin{align*}
M = \begin{bmatrix}
a & e/b \\
e/c & d
\end{bmatrix}
\end{align*}
for some $a > 0,b >0,c>0,d \geq 0, e \geq 0$ and let $\rho$ be the largest absolute eigenvalue of $M$. Then, there exists a sequence $\omega_t$ such that 
\begin{align*}
\| M^k \| = \prod_{t=0}^{k-1} (\rho + \omega_t) \quad \text{and} \quad \text{lim}_{t \rightarrow \infty} \, \omega_t = 0.
\end{align*}
\end{lemma}
\begin{proof}
The characteristic equation reads
\begin{equation*}
\det(M-\lambda I) = \lambda^2 - \lambda (a+d)  + ad - \dfrac{e^2}{bc} = 0
\end{equation*} 
with the discriminant of 
\begin{equation*}
(a-d)^2 + \dfrac{4e^2}{bc} \geq 0.
\end{equation*}
Thus, all eigenvalues are real.

First, we consider the case when $\det(M-\lambda I)=0$ has a double root. 
We obtain the condition for a double root as
\begin{equation*}
(a-d)^2 + \dfrac{4e^2}{bc} = 0.
\end{equation*}
Since $b>0$ and $c > 0$, this implies
\begin{align*}
a = d, \quad e = 0.
\end{align*}
Therefore, $M=aI$ and $\rho = a$. From $M^k = a^k I$, we have
\begin{align*}
\| M^k \| = \sqrt{a^{2k}} = \rho^k,
\end{align*}
resulting in 
\begin{equation*}
\omega_k = \dfrac{\| M^{k+1} \|}{\| M^k \|} - \rho  = \rho - \rho = 0
\end{equation*}
for all $k \geq 0$.

Next, we consider the case when $M$ has two distinct eigenvalues $\lambda_1$ and $\lambda_2$. Since $a+d > 0$, we have $\lambda_1 + \lambda_2 > 0$. Without loss of generality, assume $\lambda_1 > \lambda_2$. Then, $\rho = \lambda_1$. Let $v_1$ and $v_2$ be corresponding eigenvectors of $\lambda_1$ and $\lambda_2$, respectively. Since $v_1$ and $v_2$ are linearly independent we can represent each column of $M$ as a linear combination of $v_1$ and $v_2$ as
\begin{equation*}
M = [\alpha_1 v_1 + \beta_1 v_2 \quad \alpha_2 v_1 + \beta_2 v_2].
\end{equation*}
By repeatedly multiplying $M$, we obtain
\begin{equation*}
M^k = [\alpha_1 \lambda_1^{k-1} v_1 + \beta_1 \lambda_2^{k-1} v_2 \quad \alpha_2 \lambda_1^{k-1} v_1 + \beta_2 \lambda_2^{k-1} v_2].
\end{equation*}
Let $C^k = (M^k)^TM^k$. Then, we have
\begin{align*}
C^k_{11} = \alpha_1^2 \lambda_1^{2(k-1)} + \beta_1^2 \lambda_2^{2(k-1)} + 2\alpha_1 \beta_1 (\lambda_1 \lambda_2)^{k-1} v_1^Tv_2 \\
C^k_{22} = \alpha_2^2 \lambda_1^{2(k-1)} + \beta_2^2 \lambda_2^{2(k-1)} + 2\alpha_2 \beta_2 (\lambda_1 \lambda_2)^{k-1} v_1^Tv_2 
\end{align*}
and
\begin{align*}
C^k_{12} = \alpha_1 \alpha_2 \lambda_1^{2(k-1)} + \beta_1 \beta_2 \lambda_2^{2(k-1)} + (\alpha_1 \beta_2 + \alpha_2 \beta_1) (\lambda_1 \lambda_2)^{k-1} v_1^Tv_2, \quad C^k_{21} = C^k_{12}.
\end{align*}
Since
\begin{align*}
C^k_{11} \geq 
\alpha_1^2 \lambda_1^{2(k-1)} + \beta_1^2 \lambda_2^{2(k-1)} - 2\alpha_1 \beta_1 (\lambda_1 \lambda_2)^{k-1}
=
\left( \alpha_1 \lambda_1^{k-1} - \beta_1 \lambda_2^{k-1} \right)^2 \geq 0
\end{align*}
and 
\begin{align*}
C^k_{22} \geq 
\alpha_2^2 \lambda_1^{2(k-1)} + \beta_2^2 \lambda_2^{2(k-1)} - 2\alpha_2 \beta_2 (\lambda_1 \lambda_2)^{k-1}
=
\left( \alpha_2 \lambda_1^{k-1} - \beta_2 \lambda_2^{k-1} \right)^2 \geq 0,
\end{align*}
we have
\begin{align*}
\| M^k \| = \sqrt{\frac{1}{2} 
\left[
C_{11}^k + C_{22}^k + \sqrt{\left( C_{11}^k - C_{22}^k \right)^2 + 4(C_{12}^k)^2} 
\right]},
\end{align*}
leading to
\begin{align*}
\dfrac{\| M^{k+1} \|}{\| M^k \|} = \sqrt{\frac{C_{11}^{k+1} + C_{22}^{k+1} + \sqrt{\left( C_{11}^{k+1} - C_{22}^{k+1} \right)^2 + 4(C_{12}^{k+1})^2}}
{C_{11}^k + C_{22}^k + \sqrt{\left( C_{11}^k - C_{22}^k \right)^2 + 4(C_{12}^k)^2}}}.
\end{align*}
From
\begin{equation*}
\lim_{k \rightarrow \infty} \dfrac{C_{11}^k}{\lambda_1^{2(k-1)}} = \alpha_1^2, \quad \lim_{k \rightarrow \infty} \dfrac{C_{22}^k}{\lambda_1^{2(k-1)}} = \alpha_2^2, \quad \lim_{k \rightarrow \infty} \dfrac{C_{12}^k}{\lambda_1^{2(k-1)}} = \lim_{k \rightarrow \infty} \dfrac{C_{21}^k}{\lambda_1^{2(k-1)}} = \alpha_1 \alpha_2,
\end{equation*}
we obtain
\begin{equation*}
\lim_{k \rightarrow \infty} \dfrac{\| M^{k+1} \|}{\| M^k \|} = \sqrt{\lambda_1^2} = \rho.
\end{equation*}
From
\begin{equation*}
\lim_{k \rightarrow \infty} \omega_k = \lim_{k \rightarrow \infty} \dfrac{\| M^{k+1} \|}{\| M^k \|} - \rho = \rho - \rho = 0,
\end{equation*}
we obtain the desired result.
\end{proof}





\vskip 0.2in
\newpage
\bibliographystyle{plainnat}
\bibliography{main}

\end{document}